\documentclass[a4paper,11pt,reqno]{amsproc}
\usepackage[utf8]{inputenc}
\usepackage{latexsym}\usepackage{ifthen}
\usepackage{enumerate}\usepackage{calc}\usepackage{hyphenat}
\usepackage{amstext,amsbsy,amsopn,amsthm,amsgen}
\usepackage{amsfonts,amssymb,amscd,amsxtra,upref}
\usepackage{mathrsfs}\usepackage{euscript}
\usepackage{graphicx,color}\usepackage{verbatim,hyperref}
\usepackage[T1]{fontenc}
\theoremstyle{plain}
\makeatletter\@namedef{subjclassname@2020}{\textup{2020} Mathematics Subject Classification}\makeatother
\usepackage{chngcntr}
\newtheorem{Thm}{Theorem}[section]
\newtheorem{Lem}[Thm]{Lemma}
\newtheorem{Cor}[Thm]{Corollary}
\newtheorem{Pro}[Thm]{Proposition}

\theoremstyle{definition}
\newtheorem{Def}[Thm]{Definition}
\newtheorem{Przyk}[Thm]{Example}

\theoremstyle{remark}
\newtheorem{Rem}[Thm]{Remark}
\numberwithin{equation}{section}
\ifthenelse{\isundefined{\texorpdfstring}}{\newcommand{\texorpdfstring}[2]{#1}}{}

\newenvironment{abece}[1][\em]{\begin{enumerate}[#1\ (a)]}{\end{enumerate}}
\newenvironment{property}[1]{\begin{equation}%
\label{eq:#1}\begin{minipage}[c]{.85\textwidth}\itshape}{\end{minipage}\end{equation}\ignorespacesafterend}
\newenvironment{property*}{\begin{displaymath}\begin{minipage}[c]{.9\textwidth}\itshape}{\end{minipage}\end{displaymath}\ignorespacesafterend}

\newcommand{\refoneq}[2][=]{\overset{\eqref{#2}}{#1}} \newcommand{\spacedeq}[2][=]{\hspace{#2}#1\hspace{#2}}
\newcommand{\rest}[1]{{\raise-.3ex\hbox{\big|}}_{#1}} 
\newcommand*{\ud}{\mathrm{\,d}}\newcommand{\ZP}{{\mathbb{Z}_+}}
\newcommand{\R}{\mathbb{R}}\newcommand{\C}{\mathbb{C}}\newcommand{\N}{\mathbb{N}}
\newcommand{\Borel}{\mathfrak{B}}
\newcommand{\pImp}[2]{\par(\ifthenelse{\equal{#2}{#1}}%
{$\Rightarrow$}{#1)$\Rightarrow$(#2})}
\newcommand{\lImp}[2]{\par(\ifthenelse{\equal{#2}{#1}}%
{$\Leftarrow$}{#1)$\Leftarrow$(#2})}
\newcommand{\NWSR}{the following conditions are equivalent:}
\newcommand{\BO}[1]{\boldsymbol{B}(#1)} \newcommand{\Hilb}{\mathcal{H}} 
\newcommand{\scp}[1]{\left\langle#1\right\rangle} 
\newcommand{\dom}{\mathscr{D}} 
 \newcommand{\eld}[1]{\ell^2(#1)}
\newcommand{\Tree}{\mathcal{T}}\newcommand{\Troot}{\mathsf{root}}
\newcommand{\parf}{\mathsf{p}}\newcommand{\Des}{\mathsf{Des}}
\newcommand{\Deso}{\Des^\circ}
\newcommand{\Chif}[1][]{\mathsf{Chi}\ifthenelse{\equal{#1}{}}{}{^{\langle #1\rangle}\!}}
\newcommand{\Chit}[2]{\mathsf{Chi}^{\langle #1\rangle}_{#2}}
\newcommand{\hip}[1]{\mathtt{h}^{(#1)}}\newcommand{\sqlow}[1]{_{\langle#1\rangle}}
\newcommand{\rootsum}{\mathop{ \makebox[1em]{ \makebox[0pt]{\ \ \,\scalebox{2}{$.$}}$\bigsqcup$ } }}
\newcommand{\blambda}{{\boldsymbol\lambda}}
\newcommand{\tlambda}{ {\boldsymbol{\tilde\lambda}} }
\title[Joint backward extension property]{Joint backward extension property\\for weighted shifts on directed trees} 
\author[P.\ Pikul]{Piotr Pikul}

\allowdisplaybreaks[3]
\address{P. Pikul\\Instytut Matematyki\\
 Wydzia\l{} Matematyki i~Informatyki\\Uniwersytet Jagiello\'{n}ski\\
 ul.\ \L{}ojasiewicza 6\\30-348 Krak\'{o}w\\Poland}
\email{Piotr.Pikul@im.uj.edu.pl}
\begin{document}
\begin{abstract}
Weighted shifts on directed trees are a decade-old generalisation of classical
shift operators in the sequence space $\ell^2$.
In this paper we introduce the joint backward extension property (JBEP) for
classes of weighted shifts on directed trees. If a class satisfies JBEP, the
existence of a common backward extension within the class for a family of
weighted shifts on rooted directed trees does not depend on the additional
structure of the big tree (of fixed depth).
We decide whether several classes of operators have JBEP. For subnormal or
power hyponormal weighted shifts the property is satisfied, while it fails
for completely hyperexpansive or quasinormal. Nevertheless, some positive
results on joint backward extensions of completely hyperexpansive weighted
shifts are proven.
\end{abstract}
\maketitle
\tableofcontents

\section{Introduction} 
The weighted shifts on $\ell^2$ are long known objects in the operator theory.
A survey on them can be found e.g.\ in \cite{shields}.
The study of backward extensions (sometimes called \emph{back-step} extensions)
was originally about whether or not the sequence $\{ a_n \}_{n=0}^\infty$ of
weights can be extended by a prefix $\{a_{-n}\}_{n=1}^{k}$ so that the weighted
shift with weights $\{ a_{n-k} \}_{n=0}^\infty$ has specified property (e.g.\ 
if it is subnormal). For the results on this topic we refer the reader to e.g.\ 
\cite{curtoq,back-ext2,back-ext-sub,hypexp-ext,back-ext3,%
back-ext1,fhs-momenty,fhs-posdef}.

In \cite{szifty} there was introduced a generalisation of the concept of a
weighted shift, namely a weighted shift on a directed tree. To state the idea
briefly, dealing with trees we allow an element of orthonormal basis to
have more than one successor (child). Formula defining action of such
operator takes the form (cf.\ \eqref{eq:naBazowym})
\[ S_\blambda e_v = \sum_{u\in\Chif(v)} \lambda_u e_u,\]
where $\Chif(v)$ is the set of children of the vertex $v$.

The study of backward extensions of weighted shifts on directed trees
was initiated in \cite{pikul} and this paper extends some results presented
therein. In the class of weighted shifts on directed trees we have great
freedom in defining a backward extension. Especially we can ask whether
a family of weighted shifts on rooted directed trees has a common backward
extension with certain property. The simplest case of such a joint extension
involves only one new vertex, a parent of all the former roots
(see \emph{rooted sum}, Definition~\ref{dfn:rooted-sum}), and weights
corresponding to new edges.

We can also consider more complex additional structure of the ``extended tree''.
In case of several classes of operators, given at most countable uniformly
bounded family of weighted shifts on rooted directed trees admitting
``$k$-step backward extensions'' (see Definition~\ref{dfn:backward-ext})
we can extend them all to a weighted shift on a single rooted directed tree,
regardless of the first $k$ ``levels'' of the extended tree
(see Theorem~\ref{thm:joint-extension}). This behaviour is a consequence
of the \emph{joint backward extension property} (see Section~\ref{sec:JBEP}).
The property is satisfied e.g.\ for power hyponormal or subnormal weighted shifts
while it fails to hold for completely hyperexpansive or quasinormal.

Section~\ref{R:lasy} recalls basic facts about weighted shifts on directed trees
and basic operations on said trees. While the notation is based mainly on
the directed forest approach presented in \cite{pikul}, it should not be
confusing to the reader familiar with \cite{szifty}.

In Section~\ref{sec:JBEP} we introduce concepts related to backward extensions.
Especially the formulation of joint backward extension property (JBEP) is given
(see Definition~\ref{dfn:ext-prop}) and its major consequence is proven. 
While we can extend the given collection of trees backwards by $k$ steps
in various different ways, the choice of the resulting tree
is irrelevant for the question of the existence of a weighted shift
extending the given collection of weighted shifts and belonging to
a class that has the JBEP (see Theorem~\ref{thm:joint-extension}).
There are given some examples of well known classes for which it can be easily
verified whether they satisfy JBEP.

Then we answer the questions whether the classes of subnormal, power hyponormal
(Section~\ref{sec:powhyp1}) and completely hyperexpansive (Section~\ref{sec:che1})
weighted shifts have the JBEP.
The answer is affirmative for the first two classes (see
Theorems~\ref{thm:sub-JBEP} and \ref{thm:powhyp-JBEP}) but not for the third one
(see Example~\ref{ex:che-nonjoint}). Nevertheless, in the last section we provide
some positive results about the existence of joint backward extensions for
completely hyperexpansive weighted shifts (see Theorems~\ref{thm:che-rootsum}
and \ref{thm:che-joint2}).

We denote by $\N$ and $\ZP$ the sets of positive and non-negative integers, respecively.
Standard symbols $\R$ and $\C$ stand for the sets of real and complex numbers.
We will use the disjoint union symbol ``$\sqcup$'' to emphasise
the disjointness of considered sets.
Given a topological space $X$, by $\Borel(X)$ we mean the $\sigma$-algebra of
Borel subsets of $X$. For $t\in X$ we denote by $\delta_t$ the Borel probability
measure on $X$ concentrated on $\{t\}$ (Dirac measure).
We follow the conventions that $0^0=1$ and $\sum_{v\in\emptyset} x_v=0$.
If $f\colon A\to A$ is any function, then $f^0=\mathrm{id}_A$.
For any $k\in\N$, by convention the function $t\mapsto t^{-k}=\frac{1}{t^k}$ attains
$+\infty$ at $0$.

Throughout this paper all Hilbert spaces will be over the complex numbers.
For a Hilbert space $\Hilb$ we denote the algebra of all bounded linear
operators on $\Hilb$ by $\BO{\Hilb}$.

We recall that an operator $T\in \BO{\Hilb}$ is said to be \emph{subnormal}
if it is a restriction of some normal operator $N\in\BO{\mathcal K}$ (on a possibly
larger Hilbert space $\mathcal K$) to $\Hilb$. Clearly $\Hilb$ is an invariant
subspace for $N$. It is called \emph{hyponormal} if $[T^*, T]:=T^*T-TT^*\geq 0$.
The following property
of hyponormal operators is well known (see \cite[Proposition 4.4 (e)]{Con}).
\begin{property}{hypo-inv}
If $T\in\BO{\Hilb}$ is hyponormal and $\mathcal K$ is a closed
invariant subspace for $T$, then $T\rest{\mathcal K}$ is hyponormal.
\end{property}
Notions of subnormality and hyponormality originate from the work of Halmos
\cite{halmos}. We refer the reader to \cite{Con} for the fundamentals of
the theory of subnormal and hyponormal operators.

We will also study the question of when an operator $T$ is \emph{power hyponormal},
i.e.\ whether $T^n$ is hyponormal for every $n\geq 1$. It is well known that every
subnormal operator is power hyponormal but the converse in not true.
The early counterexample is due to Stampfli \cite{stamp}.

\section{Weighted shifts on directed trees}\label{R:lasy}
Weighted shifts on directed trees were introduced in \cite{szifty}. Foundations given
therein are sufficient to understand the operators studied in this paper. However,
we will use notation focused on directed forests from \cite{pikul}. In this
section we recall basic facts. The proofs can be found either in
\cite[Section~2]{pikul} or derived from the very similar arguments presented in
\cite{szifty}.

\subsection{Directed trees}
This subsection contains graph-theoretical part of the article. We recall definitions
and basic properties. While throughout this paper we will consider directed trees
exclusively, we start with a slightly more general notion of a directed forest.
\begin{Def} Pair $\Tree=(V,\parf)$ is called a~\emph{directed forest} if
$V$ is a~nonempty set and $\parf\colon V\to V$ is a function satisfying
the following condition:
\begin{property*}
if $n\geq 1$, $v\in V$ and $\parf^n(v)=v$, then $\parf(v)=v$.
\end{property*}
Elements of the set $\Troot(\Tree):=\{u\in V\colon \parf(u)=u\}$ are called
\emph{roots} of the forest $\Tree$. We call elements of the set $V$ \emph{vertices}
of the forest and $\parf$ is called the \emph{parent function}.
\end{Def}

Later in this section we assume that $\Tree=(V,\parf)$ is a directed forest.

If  $u,v\in V$ are vertices such that $\parf(v)=u\neq v$, then we call $u$ the
\emph{parent of $v$}. We also set $V^\circ:=V\setminus \Troot(\Tree)$.
The elements of $\Chif(v):=\parf^{-1}(v)\setminus \{v\}$
are called \emph{children of the vertex $v$}.
In general, we define the set of \emph{$k$-th children of $v$} as
\begin{equation*}
\Chif[k](v):=\{u\in V\colon \parf^k(u)=v\neq \parf^{k-1}(u)\},\quad k\geq 1,
\end{equation*}
and $\Chif[0](v):=\{v\}$. Clearly $\Chif(v)=\Chif[1](v)$.
We also denote the set of all \emph{descendants} of a vertex $v$
by $\Des(v):=\bigcup_{n=0}^\infty \Chif[n](v)$ and use the symbol $\Deso(v)$ for
$\Des(v)\setminus\{v\}$.
The cardinality of the set $\Chif(v)$ is called \emph{the degree} of the vertex $v$
and denoted by $\deg(v)$.

If there is a risk of ambiguity we write $\Chif_\Tree(v)$, $\Des_\Tree(v)$, etc.\ 
to make the dependence on $\Tree$ explicit.

An element of $V\setminus \parf(V)$ is called a \emph{leaf}\footnote{Note that if
a root has no children (it is not a child of itself) it still
cannot be called a `leaf'.}.
If there is no leaf (i.e.\ $\parf(V)=V$) then the forest is called \emph{leafless}.

In the following lemma we list some basic facts.
Their elementary proofs are omitted.
\begin{Lem}\label{lem:tree-basics} Let $\Tree=(V,\parf)$ be a directed forest. Then
\begin{abece}
\item\label{tbs:child-sep} $\Chif[k](v)\cap\Chif[k](w)=\emptyset$
for $k\in\ZP$ and distinct $v,w\in V$,

\item\label{tbs:kChi} $\Chif[k](v)\cap\Chif[l](v)=\emptyset$
for $k,l\in\ZP$, $v\in V$, provided that $k\neq l$,


\item\label{tbs:dzieci} for any $k\in \N$,
\[\Chif[k](v)=\bigsqcup_{u\in \Chif(v)} \Chif[k-1](u)
	= \bigsqcup_{u\in \Chif[k-1](v)} \Chif(u),\]
	

\item\label{tbs:Des-Des} if $v\in V$ and $w\in\Des(v)$,
then $\Des(w)\subseteq\Des(v)$,\vspace{.5ex}

\item\label{tbs:leafs} $v\in V$ is a leaf if and only if $v\in V^\circ$ and $\deg(v)=0$.
\end{abece}
\end{Lem}

A \emph{tree} in the forest $\Tree$ is an equivalence class in $V$ with
respect to the relation
\[ v\sim w \colon \iff \exists_{m,n\in\ZP}\ \parf^n(v)=\parf^m(w). \]
In other words, a tree is a connected component of the graph
represented by $\Tree$. Some properties of trees in directed forests
can be found in \cite[Lemma~2.5]{pikul}.

A~\emph{directed tree} is a directed forest containing only one tree.
In other words, it is a driected forest which is connected as a~graph.
The forest is called \emph{degenerate} if $V^\circ = \emptyset$
(i.e.\ $\parf=\mathrm{id}_V$). Otherwise it is non-degenerate. It may be worth
noting that a non-degenerate directed tree is leafless if and only if every
vertex has a positive degree.

\begin{Rem}\label{rem:directed-trees}
Our definition of a directed tree is consistent with the ``traditional'' one
(see e.g.\ \cite{szifty}). The corresponding set of edges can be defined as
(cf.\ \cite[(3.2.1)]{budz})
\[ E_\Tree:=\bigl\{ (u,v)\in V^2\colon u\neq v,\ u=\parf(v) \bigr\}. \]
In particular, the notions of a~root, a~child and a~$k$-th child
coincide (cf.\ Lemma \ref{lem:tree-basics}~\eqref{tbs:dzieci}).
Main differences regard the parent function (in our case defined also
on a~root) and the definition of leaves (we never call the root
a leaf). 
\end{Rem}

For every tree $T\subseteq V$ in a directed forest $\Tree=(V,\parf)$, the pair
$(T,\parf\rest{T})$ is a~directed tree and the set $\Troot\bigl((T,\parf\rest{T})\bigr)$
has at most one element. 
A~directed tree with a~root is called a \emph{rooted directed tree}.
We usually denote the root of a directed tree $\Tree$ by $\omega_\Tree$,
or simply $\omega$ if there is no ambiguity.
A~directed tree without a root is called \emph{rootless}.

For $W\subseteq V$ we can define a~\emph{subforest}
$\Tree\rest{W}=(W,\parf_W)$, where
\[ \parf_W(v):=\begin{cases}\parf(v) & \text{if }\parf(v)\in W\\
v& \text{if }\parf(v)\notin W\end{cases},\quad v\in W.\]
If $W=\Des(v)$ for some $v\in V$ we write $\Tree_{(v\to)}:=\Tree\rest{W}$.
\label{dfn:Des-tree}
Then $\Tree_{(v\to)}$ is a rooted directed tree and $\Troot(\Tree_{(v\to)})=\{v\}$.
In this case we will use the name \emph{subtree}.

For directed forests a natural notion of isomorphism can be introduced
(see \cite[Definition 2.8]{pikul}). In most cases we are not making any
distinction between isomorphic forests.
Forests are said to be disjoint if so are their sets of vertices.
We can always assume that the considered forests are disjoint by
taking appropriate isomorphic forests if necessary.

\subsection{Backward extensions of directed trees}
In the following subsection we recall operations on directed trees that are
essential for the study of backward extensions of associated weighted shift operators.

Let us start with the simplest backward extension of a directed tree.

\begin{Def}\label{dfn:backward-ext}
Given a directed tree $\Tree=(V,\parf)$ with root $\omega$ and $k\geq 0$ we define
its \emph{$k$-step backward extension} as $\Tree\sqlow{k}:=\bigl(V\sqlow{k}, \parf\sqlow{k} \bigr)$,
where $V\sqlow{k}:=V\sqcup\{\omega_j\}_{j=1}^k$, $\omega_j$'s are new distinct vertices,
$\omega_0:=\omega$ and 
\[ \parf\sqlow{k}(v):=\begin{cases}\parf(v)&\text{if }v\in V^\circ\\
\omega_{j}&\text{if }v=\omega_{j-1},\ j=1,\ldots,k\\\omega_k&\text{if }v=\omega_k\end{cases}
,\quad v\in V\sqlow{k}.\]
\end{Def}

\begin{figure}[ht]
\includegraphics[scale=1.1]{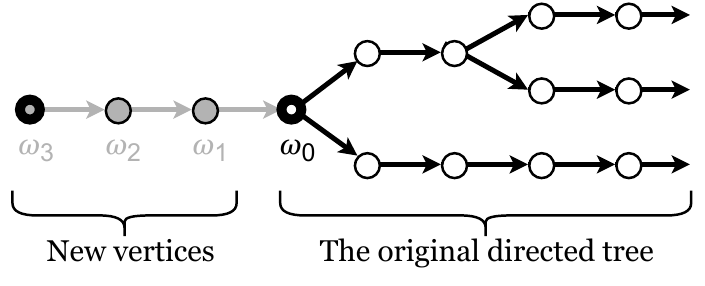}
\caption[$3$-step backward extension of a directed tree]%
{$3$-step backward extension of a directed tree.}
\label{fig:back-ext}
\end{figure}

We will call any directed tree isomorphic to the one constructed above
a $k$-step backward extension of $\Tree$. 
Note also that $\Tree\sqlow{0}=\Tree$.

Another important operation on rooted directed trees is joining them by adding
a new root as a parent for their original roots. Below you can find the
formal definition. See also Figure~\ref{fig:rooted-sum}.

\begin{Def}\label{dfn:rooted-sum}
The \emph{rooted sum} of a~family $\{\Tree_j=(V_j,\parf_j)\colon j\in J\}$ of
pairwise disjoint rooted directed trees is the directed tree defined by
\[\rootsum_{j\in J}\Tree_j := \bigl(V_J, \parf_J\bigr), \]
where $\displaystyle
V_J:=\{\omega\}\sqcup\bigsqcup_{j\in J}V_j$ and
\[
\displaystyle \parf_J(v):=\begin{cases}\parf_j(v) & \text{if }v\in V^\circ_j\\
\omega &\text{otherwise}\end{cases},\quad v\in V_J.\]
\end{Def}
\begin{figure}[hbt]
\includegraphics[scale=.8]{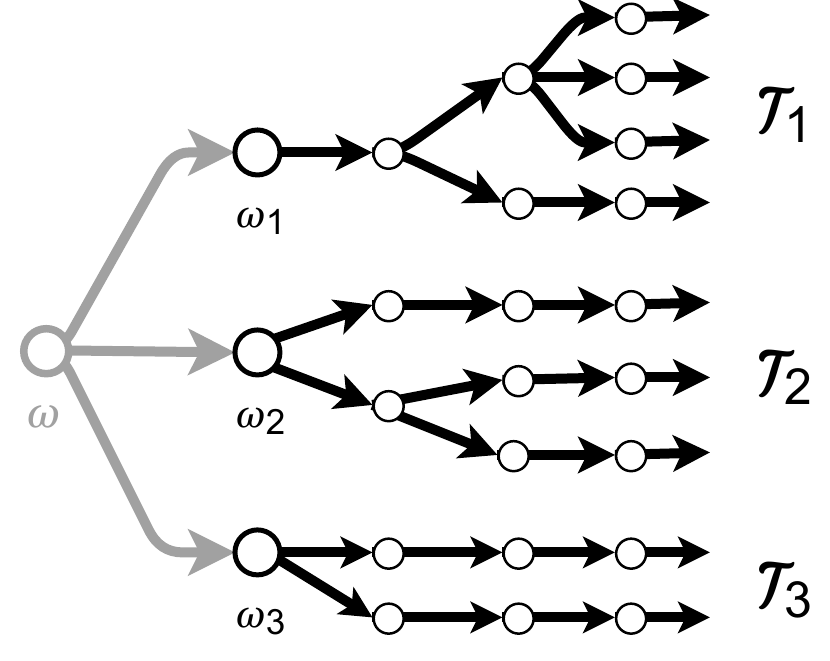}
\caption[Rooted sum of directed trees]%
{The rooted sum of the family $\{\Tree_1,\Tree_2,\Tree_3\}$.}
\label{fig:rooted-sum}
\end{figure}

Note that if the given family is empty, its rooted sum contains only
one vertex and is a degenerate directed forest (tree).
Similarly as before, we will call any directed tree isomorphic to
$\rootsum_{j\in J}\Tree_j$ a rooted sum of the family $\{\Tree_j\}_{j\in J}$.
In particular 
\begin{property}{child-sum}
for any directed forest $\Tree$ and its vertex $v$, 
the directed tree $\Tree_{(v\to)}$ is a rooted sum of the family
$\{\Tree_{(u\to)}\colon u\in\Chif(v)\}$.
\end{property}

\subsection{Weighted shift operators} 
\label{R:szifty}
This subsection recalls the weighted shift operators on directed forests as
defined is \cite[Section~2]{pikul}. Most of fundamental facts is analogous
to the case of weighted shifts on directed trees described in \cite{szifty}.

Given a~set $V$ we consider the Hilbert space $\eld{V}$ of all square-summable
complex functions on $V$ with the standard inner product
\[ \scp{f,g}:=\sum_{v\in V} f(v)\overline{g(v)},\quad f,g\in\eld{V}.\]
By $\{e_v:v\in V\}$ we denote the canonical orthonormal basis of $\eld{V}$
consisting of functions such that $e_v\rest{V\setminus\{v\}}\equiv 0$ and $e_v(v)=1$.

For a set $W\subseteq V$, we will identify $\eld{W}$ with a closed subspace of
$\eld{V}$ obtained by taking extensions of elements from $\eld{W}$ by $0$ on
$V\setminus W$. In particular $\eld{\emptyset}=\{ 0 \}$.
With this convention we have $\eld{W_1} \perp \eld{W_2}$, provided that
$W_1$ and $W_2$ are disjoint subsets of $V$.
In most cases in this paper $V$ will be the set of vertices of a directed tree.

\begin{Def}[Weighted shift]\label{dfn:shift-op}
For a~directed forest $\Tree=(V,\parf)$ and a~set of complex weights
$\blambda=\{\lambda_v\}_{v\in V}$ such that $\lambda_\omega=0$ for
$\omega\in \Troot(\Tree)$,
we define the operator $S_\blambda$ in $\eld{V}$, called the 
\emph{weighted shift on $\Tree$ with weights $\blambda$}, as follows:
\[ \dom(S_\blambda):=\{f\in \ell^2(V)\colon \varGamma_\blambda(f) \in \eld{V}\}, \]
\[ S_\blambda (f):= \varGamma_\blambda(f),\quad f\in \dom(S_\blambda),\]
where $\varGamma_\blambda \colon\C^V\to \C^V$ is defined by
$\varGamma_\blambda(f)(v):= \lambda_v f(\parf(v))$.

We say a weighted shift $S_\blambda$ is \emph{proper} if
$\{v\in V\colon\lambda_v=0\}= \Troot(\Tree)$.
\end{Def}
The above definition of a weighted shift operator is almost the same as
\cite[Definition 3.1.1]{szifty}. For a directed tree both definitions of
$S_\blambda$ agree up to the requirement of specifying the zero weight
for the root.
Later in this paper the statement ``$\blambda$ is a system of weights on $\Tree$''
includes the assumption that $\lambda_\omega=0$ for $\omega\in\Troot(\Tree)$.

A weight can be understood as being assigned to an edge joining a vertex with
its parent. As stated in the equality \eqref{eq:naBazowym}, the operator
$S_\blambda$ \emph{shifts} a value at given vertex through all outgoing edges.

In view of Definition~\ref{dfn:shift-op}, a classical unilateral weighted shift
on $\eld{\ZP}$ is a weighted shift on the directed tree $\Tree=(\ZP,\parf)$,
where $\parf(0)=0$ and $\parf(n+1)=n$ for $n\geq 0$.
For a bilateral weighted shift (acting on $\eld{\mathbb Z}$) we use the tree
$\Tree=(\mathbb Z,\parf)$ with $\parf(n)=n-1$ for $n\in\mathbb Z$.

Proofs of the following simple facts does not differ from the case of directed
trees described in \cite{szifty} (cf.\ Propositions 3.1.3 and 3.1.8 therein).

\begin{Pro}\label{pro:shift-basic}
Let $S_\blambda$ be a weighted shift on a directed forest $\Tree=(V,\parf)$. Then
\begin{abece}
\item if $v\in V$ and $e_v\in\dom(S_\blambda)$, then
\begin{equation}\label{eq:naBazowym}
S_\blambda e_v = \sum_{u\in\Chif(v)} \lambda_u e_u,
\end{equation}

\item\label{shbs:bounded}  $S_\blambda$ is a bounded operator on $\eld{V}$ if and only if
\[ \sup\biggl\{\sum_{u\in\Chif(v)} |\lambda_u|^2 \colon v\in V \biggr\} < \infty; \]
moreover, if $S_\blambda \in \BO{\eld{V}}$, then
\[\|S_\blambda\|^2=\sup\bigl\{\|S_\blambda e_v\|^2\colon v\in V\bigr\}
=\sup\biggl\{ \sum_{u\in\Chif(v)} |\lambda_u|^2 \colon v\in V \biggr\},\]
\item\label{shbs:below}  $S_\blambda$ is bounded from below on $\dom(S_\blambda)$ 
if and only if
\[ \inf\biggl\{\sum_{u\in\Chif(v)} |\lambda_u|^2 \colon v\in V \biggr\} > 0, \]
\item\label{shbs:countable} if $S_\blambda$ is proper and densely defined,
then $\Tree$ is locally countable $($i.e.\ every tree in $\Tree$ is countable$)$,
\item\label{shbs:Des-inv} if $\eld{\Des(v)}\subseteq \dom(S_\blambda)$
for some $v\in V$, then \[ S_\blambda(\eld{\Des(v)})\subseteq \eld{\Des(v)}. \]
\end{abece}
\end{Pro}


To describe powers of a~weighted shift we introduce the following notation
\[\lambda^{(k)}_v:=\begin{cases} 1 & 	 $if $ k=0\\ 
\prod_{j=0}^{k-1} \lambda_{\parf^j(v)} & $if $ k\geq 1\end{cases},
\qquad\ v\in V.\]
Then the following fact holds (cf.\ \cite[Lemma~3.6]{pikul}).
\begin{Lem}\label{lem:powers} Suppose that $S_\blambda$ is a bounded
weighted shift on a directed forest $\Tree=(V,\parf)$. Let $k\in\N$.
Then for $f\in\eld{V}$,
\begin{eqnarray}\setlength\arraycolsep{0pt}
\label{eq:Sdok} S_\blambda^k e_v &=&
\sum_{u\in\Chif[k](v)}\lambda^{(k)}_u e_u,\\
\label{eq:S*dok} S_\blambda^{\ast k} e_v &=&
\overline{\lambda^{(k)}_v} e_{\parf^k(v)}.
\end{eqnarray}
\end{Lem}

A benefit from considering weighted shifts on directed forests is that
this class is closed on taking powers of operator. In fact, the operator
$S_\blambda^k$ is the weighted shift on the $k$-th power $\Tree^k$ of the
forest $\Tree$ with weights $\blambda^{(k)}=\{\lambda_v^{(k)}\}_{v\in V}$.
The reader is referred to \cite[Definition~2.12]{pikul} for the notion of
$k$-th power of a directed forest.
According to \cite[Lemma~2.14~(c)]{pikul}, if $S_\blambda$ is a~weighted shift
on an infinite directed tree, $S_\blambda^k$ is a~weighted shift on a~directed
forest consisting of at least $k$ trees. In particular, for $k\geq 2$, it is
no longer a~weighted shift on a directed tree.

As a consequence of the formula \eqref{eq:Sdok} and
Lemma~\ref{lem:tree-basics}~(\ref{tbs:child-sep}) we obtain the following useful property.
\begin{Cor}\label{cor:szift-ort}
Let $S_\blambda$ be a bounded weighted shift on a directed forest $\Tree=(V,\parf)$,
$k\in\ZP$ and $\sum_{v\in V} a_v e_v $ be a vector in $\eld{V}$. Then
\begin{equation}\label{eq:normasumy}
\Bigl\| S_\blambda^k \sum_{v\in V} a_v e_v \Bigr\|^2 = \sum_{v\in V} |a_v|^2 \|S_\blambda^k e_v\|^2.
\end{equation}
\end{Cor}

Similarly as for the classical weighted shifts (see
\cite[Corollary 1,\ p.\,52]{shields}), the study of weighted shifts
on directed trees can be restricted to the case when all the weights
are non-negative (see \cite[Theorem~3.2.1]{szifty}, \cite[Theorem~3.7]{pikul}).
In view of \cite[Proposition~3.2]{pikul}, considering directed forests allows
us to assume that all weights (except those associated to the roots)
are strictly positive (cf.\ \cite[Proposition~3.1.6]{szifty}).

\section{Joint backward extension property} \label{sec:JBEP}
\subsection{Introduction to backward extensions} 
Given a directed tree $\Tree=(V,\parf)$ and complex numbers
$\{\lambda_v\}_{v\in \Deso(u)}$ for some $u\in V$, we will use the following notation
\begin{equation}\label{eq:destree}
\blambda^\to_u:=\{\tilde\lambda_v\}_{v\in\Des(u)},\quad
\tilde\lambda_v:=\begin{cases}\lambda_v &\text{if }v\neq u\\ 0& \text{if }v= u\end{cases},\quad v\in \Des(u).
\end{equation}

The so defined $\blambda^\to_u$ is a system of weights on the rooted directed
tree $\Tree_{(u\to)}$ (see page \pageref{dfn:Des-tree}). If $S_\blambda$ is a
bounded weighted shift on $\Tree$, then $S_{\blambda^\to_u}$, the weighted shift
on $\Tree_{(u\to)}$, is equal to the restriction $S_\blambda\rest{\eld{\Des(u)}}$
of $S_\blambda$ to the invariant subspace $\eld{\Des(u)}$
(see Proposition~\ref{pro:shift-basic}~\eqref{shbs:Des-inv}).

We will consider whether weighted shifts on rooted directed trees can be extended
to shifts on larger trees (for which the former root has a parent) with the same
properties (e.g.\ subnormality). The basic case of such extension is the subject of
the following definition.
\newcommand{\Cls}{\mathscr{C}} 

Let $\Cls$ be a class of operators, e.g.\ subnormal. We will focus on classes closed
with respect to restrictions to closed invariant subspaces.

\begin{Def}\label{dfn:backward-ext-shift}
Let $S_\blambda\in \BO{\eld{V}}$ be a weighted shift on a rooted directed tree $\Tree$.
Given $k\in\N$, we say $S_\blambda$ \emph{admits a $\Cls$ $k$-step backward extension}
if it belongs to the class $\Cls$ and there exist non-zero weights
$\lambda_{\omega_0}',\ldots,\lambda_{\omega_{k-1}}'$ such that the weighted shift
$S_{\blambda '}$ on $\Tree\sqlow{k}$ (see Definition~\ref{dfn:backward-ext})
is bounded and belongs to $\Cls$, where $\blambda '=\{\lambda_v'\}_{v\in V\sqlow{k}}$,
$\lambda'_{\omega_k}=0$ and $\lambda_v'=\lambda_v$ whenever $v\in V^\circ$.
The weighted shift $S_{\blambda'}$ is then called a $\Cls$ $k$-step backward extension
of $S_\blambda$.
We say that $S_\blambda$ \emph{admits a $\Cls$ $0$-step backward extension} 
if $S_\blambda$ belongs to $\Cls$.

In the case of a concrete class of operators we will replace ``$\Cls$'' in the above
definition by the name of members of the class, e.g.\ we will write that
a weighted shift \emph{admits a subnormal $k$-step backward extension}.
\end{Def}

Assuming the class $\Cls$ is closed on taking restrictions to invariant subspaces
(like in \eqref{eq:hypo-inv})
implies that every weighted shift on a directed tree admitting $\Cls$ $k$-step backward
extension admits also $\Cls$ $(k-1)$-step backward extension ($k\geq 1$).
Indeed, $\eld{V\sqlow{k-1}}$ is an invariant subspace for a weighted shift on
$\Tree\sqlow{k}$ (cf.\ Proposition~\ref{pro:shift-basic}~(\ref{shbs:Des-inv})).

It is worth mentioning that if a weighted shift on a directed tree admits a
subnormal (resp.\ power hyponormal) $k$-step backward extension, then any its
nonzero scalar multiple does. It is not the case in the class of completely
hyperexpansive operators (see Section~\ref{sec:che1}), because due to
$A_1(T)\leq 0$ any contractive completely hyperexpansive operator is
automatically an isometry.

\begin{Pro}[{\cite[Proposition 8.1.3.]{szifty}}] \label{pro:isometry}
Let $\Tree$ be a countable rooted and leafless directed tree. Then
there exists $\blambda$, a system of weights on $\Tree$, such that $S_\blambda$
is proper and isometric.
\end{Pro}
\begin{proof} As usual, we denote the tree by $\Tree=(V,\parf)$ and its root by $\omega$.
To construct an isometry it suffices to take nonzero weights
$\{\lambda_u\}_{u\in\Chif(v)}$ such that $\sum_{u\in\Chif(v)}|\lambda_u|^2 = 1$
for each $v\in V$.
\end{proof}

For several classes of weighted shifts the property of admitting $k$-step backward
extensions is always satisfied.
\begin{Pro}\label{pro:trivial-back-ext}
Let $\Cls$ be one of the following classes of operators
(in all cases we assume boundedness)
\begin{itemize}
\item all bounded,
\item contractive,
\item bounded from below,
\item expansive,
\item isometric,
\item scalar multiples of isometries,
\item non-zero hyponormal.
\end{itemize}
Then any weighted shift in the class $\Cls$ admits a $\Cls$ $k$-step backward
extension for arbitrary $k\in\N$.
\end{Pro}
\begin{proof}
In case of $S_\blambda$ belonging to any of the aforementioned classes,
setting $\lambda_{\omega_j}'=\|S_\blambda e_\omega\|$ for $j=0,\ldots,k-1$ gives
a desired $\Cls$ $k$-step backward extension\footnote{Extending
$0$ weighted shift in the class of bounded or contractive operators
requires}
(cf.\ Definition~\ref{dfn:backward-ext-shift}).
Indeed, for bounded and contractive operators one can use
Proposition~\ref{pro:shift-basic}~\eqref{shbs:bounded}.
For operators bounded from below or expansive c.f.\ 
Proposition~\ref{pro:shift-basic}~\eqref{shbs:below}.
Scalar multiples of isometries can be resolved similarily as in
the proof of  Proposition~\ref{pro:isometry}.
For the case of hyponormal operators see \cite[Theorem~5.1.2]{szifty} (c.f.\ 
\cite[Theorem~4.1]{pikul} with $k=1$ or Theorem~\ref{thm:pow-hypo} in this paper).
\end{proof}

It is worth recalling that according to \cite[Proposition~8.1.7.]{szifty} all
quasinormal proper weighted shifts on directed trees are scalar multiples of isometries.

\subsection{JBEP and its consequences}
A more complex case of a backward extension is passing from a family of weighted
shifts on rooted directed trees to a weighted shift on the rooted sum of given trees
altogether in the same class of operators.
For some classes the following phenomenon can happen.

\begin{Def}\label{dfn:ext-prop} 
Let $\Cls$ be a class of weighted shifts on directed trees. We say that 
$\Cls$ has the \emph{joint backward extension property (JBEP)} if
the following statement is true.\smallskip

{\centering \framebox[.95\textwidth]{\centering
\begin{minipage}{.85\textwidth}
For every $k\in\ZP$ there exists a constant $c_k>0$ such that for any
non-empty at most countable family $\{S_j\}_{j\in J}$ of bounded
proper weighted shifts on rooted directed trees $\Tree_j=(V_j,\parf_j)$,
\NWSR
\begin{abece}[]
\item there exists a system $\{ \theta_j \}_{j\in J}\subseteq \C\setminus\{0\}$
for which the weighted shift $S_\blambda$ on $\rootsum_{j\in J}\Tree_j=(V_J,\parf)$
with weights $\blambda =\{\lambda_v\}_{v\in V_J}$ such that
(with $\omega_j\in\Troot(\Tree_j)$)
\begin{equation*}
\lambda_{\omega_j}=\theta_j \text{ and }
S_{\blambda^\to_{\omega_j}} = S_j
\text{ for } j\in J,
\end{equation*}
is bounded and admits a $\Cls$ $k$-step backward extension,

\item the system $\{ \theta_j \}_{j\in J}$ in (a) can be chosen so that
additionally $\|S_\blambda\|\leq c_k\sup\{\|S_j\|\colon j\in J\}$,

\item $S_j$ admits a $\Cls$ $(k+1)$-step backward extension
for each $j\in J$ and $\sup\{\|S_j\|\colon j\in J\} < \infty$.
\end{abece} \end{minipage}}\par}
\end{Def}

For several standard classes of operators it is immediate that they have the
joint backward extension property.

\begin{Pro}The following classes of weighted shifts have the
joint backward extensions property with $c_k=1$:
\begin{abece}
\item bounded,
\item contractive $(T^*T\leq I)$,
\item expansive $(T^*T\geq I)$,
\item isometric $(T^*T= I)$,
\item hyponormal.
\end{abece}
\end{Pro}
\begin{proof}
Thanks to Proposition~\ref{pro:trivial-back-ext} it is enough to prove
that a family of weighted shifts belonging to one of the above classes
has an extension onto the rooted sum. In the case of hyponormal operators
it can be easily deduced from the condition given in \cite[Theorem~5.1.2]{szifty}.
Note that the assumption of admitting hyponormal $(k+1)$-step backward extension
implies that operators in (c) are non-zero.
Solutions for the remaining classes are straightforward.
\end{proof}

One could expect that for every class each member of which admits $k$-step
backward extension (see Proposition~\ref{pro:trivial-back-ext}) the
joint backward extension property would be satisfied. Such intuition is
misleading and an example are operators bounded from below.

\begin{Pro}\label{pro:below-fail}
The class of bounded from below weighted shifts on directed trees
does not have the joint backward extensions property.
\end{Pro}
\begin{proof} Consider a countable family $\{S_n\}_{n\in \N}$ of unilateral
weighted shifts ($\Tree_n\simeq (\N,\parf)$ for $n\in\N$, with
$\parf(n):=\max\{1,n-1\}$) such that $\|S_n e_v\|=\frac{1}{n}$ for each vertex
$v$ of $\Tree_n$. According to Proposition~\ref{pro:trivial-back-ext},
$S_n$ admits a bounded from below $k$-step backward extension for every $n,k\in \N$.
On the other hand, the operators $\{S_n\}_{n\in \N}$ clearly cannot be extended to
a single operator bounded from below.
\end{proof}
The above proof justifies that also the class of scalar multiples
of isometries (i.e.\ quasinormal weighted shifts) does not satisfy JBEP.

The joint backward extension property has an impressive consequence.
It is related to extending at most countable family of weighted
shifts on directed trees to a weighted shift on any completion
of the corresponding trees by a finite-depth\footnote{The term \emph{depth}
appears here without a formal definition, as it is not necessary in the
statement of the theorem. Figures~\ref{fig:joint-extension} and \ref{fig:two-trees}
should give an intuition.} structure preserving the class in question.

\begin{figure}[htb]\includegraphics[width=.95\textwidth]{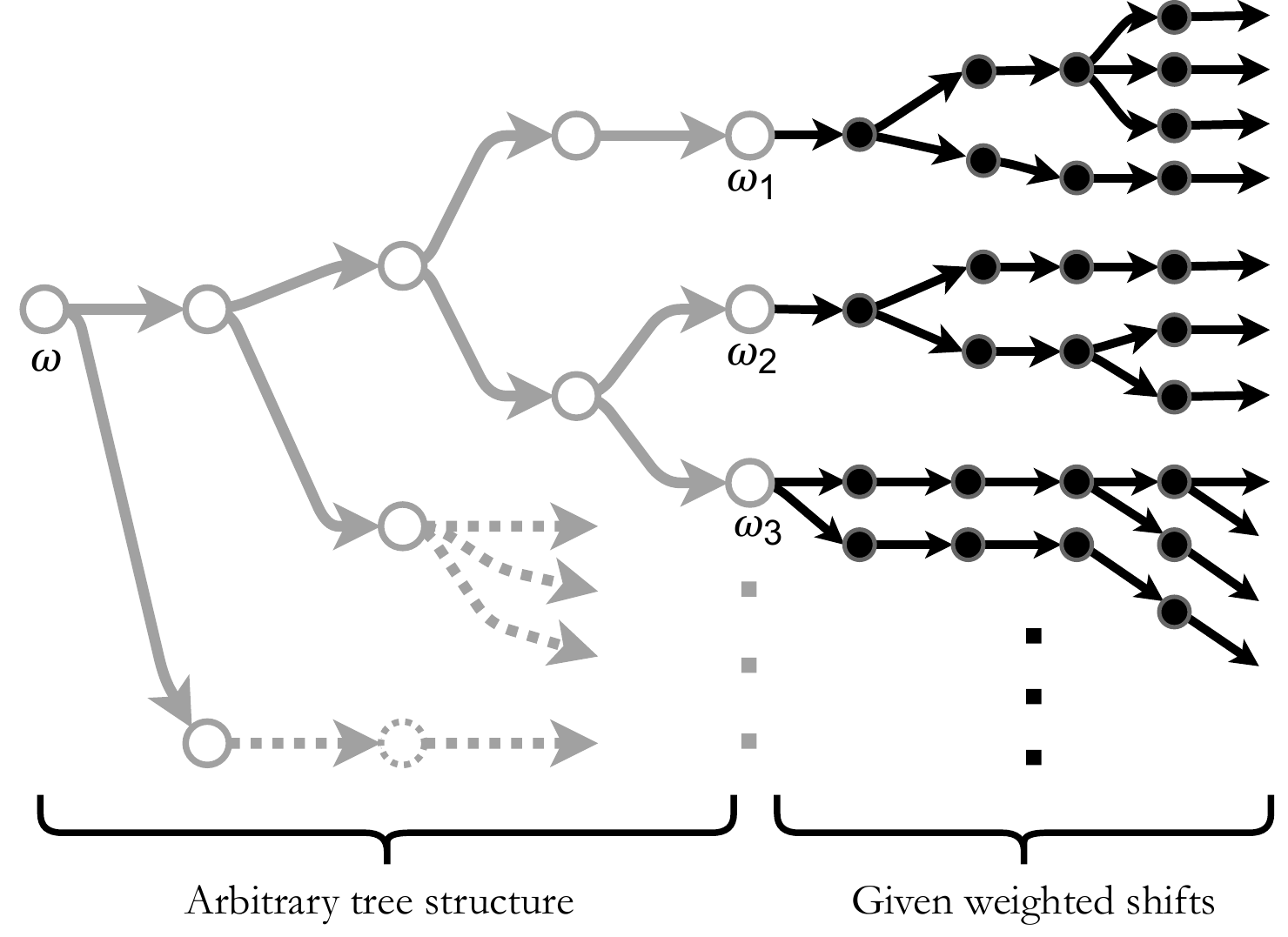}
\caption[Joint backward extension]{Subtrees joined ``at the depth $k=4$''; here
$\Chif[4](\omega)=\{\omega_1,\omega_2,\omega_3,\ldots\}$ and vertices from $W_4$
are black while those outside $W_4$ are gray.}
\label{fig:joint-extension}
\end{figure}

\begin{Thm}\label{thm:joint-extension}
Let $\Cls$ be a class with the joint backward extension property,
$\Tree=(V,\parf)$ be a leafless directed tree with the root $\omega$ and
$k\geq 1$. Assume that weights
$\{\lambda_v\}_{v\in W_k}\subseteq \C\setminus\{0\}$ are given, where
\[ W_k:=\bigcup_{v\in\Chif[k](\omega)} \Deso(v). \]
Then \NWSR \begin{abece}
\item
there exist $\{\lambda_v\}_{v\in V\setminus W_k} \subseteq \C$ such that
the weighted shift $S_\blambda$ with weights $\blambda=\{\lambda_v\}_{v\in V}$
is bounded proper and belongs to $\Cls$,
\item
the set $\Chif[k](\omega)$ is at most countable,
the weighted shift $S_{\blambda^\to_v}$ $($see \eqref{eq:destree}$)$ is 
bounded proper and admits a $\Cls$ $k$-step backward extension
for every $v\in\Chif[k](\omega)$,
and $\sup\{\|S_{\blambda^\to_v}\|\colon v\in \Chif[k](\omega)\} <\infty$.
\end{abece}
\end{Thm}
It may be worth noticing that $W_k= \{v\in V\colon \parf^k(v)\neq\omega\}$.

\begin{proof}
By Proposition~\ref{pro:shift-basic}~(\ref{shbs:countable}) and the assumption
that $\Tree$ is leafless any of the conditions (a) or (b) implies that the
entire tree (i.e.\ the set $V$) is at most countable.

We will proceed by induction on $k$.

The case $k=1$ is covered by equivalence (a)$\Leftrightarrow$(c) in
Definition~\ref{dfn:ext-prop} with $k=0$.
Now assume that (a) and (b) are equivalent for a fixed $k\geq 1$. We will show
that (a)$\Leftrightarrow$(b) for $k+1$

\pImp{b}{a} Assume that for each $v\in \Chif[k+1](\omega)$, the weighted shift
$S_{\blambda^\to_v}$ is bounded by a common constant $C>0$ and admits a $\Cls$ 
$(k+1)$-step backward extension. For a fixed $u\in \Chif[k](\omega)$ we can apply
the joint backward extension property to the family
$\{S_{\blambda^\to_v}\}_{v\in\Chif(u)}$ in order to get nonzero weights
$\{\lambda_v\}_{v\in \Chif(u)}$ such that the weighted shift $S_{\blambda^\to_u}$
admits a $\Cls$ $k$-step backward extension (see \eqref{eq:child-sum}).
According to the condition (b) in Definition~\ref{dfn:ext-prop} the weights can
be chosen such that the norm of $S_{\blambda^\to_u}$ is bounded by $c_kC$
for all $u\in \Chif[k](\omega)$.

Summarising, $\{ S_{\blambda^\to_u} \}_{u\in \Chif[k](\omega)}$ is a uniformly
bounded family of proper weighted shifts, each of which admits a $\Cls$ $k$-step
backward extension.
Using the induction hypothesis completes the proof of (a).

\pImp{a}{b}
Now assume that there exist weights $\{\lambda_v\}_{v\in V\setminus W_{k+1}}$
such that $S_\blambda$ is bounded proper and belongs to $\mathscr{C}$. We want to
show that each weighted shift $S_{\blambda^\to_v}$ with $v\in\Chif[k+1](\omega)$
admits $\Cls$ $(k+1)$-step backward extension. They are uniformly bounded as
restrictions of the bounded operator $S_\blambda$ to invariant subspaces
(see Proposition~\ref{pro:shift-basic}~\eqref{shbs:Des-inv}).

The whole system $\blambda$ can be seen as an extension of the
system $\{\lambda_v\}_{v\in W_k}$ since $W_k\supseteq W_{k+1}$.
Then, by induction hypothesis, 
each weighted shift $S_{\blambda^\to_u}$
for $u\in\Chif[k](\omega)$ admits a $\Cls$ $k$-step backward extension.
If $u\in\Chif[k](\omega)$, then $S_{\blambda^\to_u}$ is a weighted shift
on a rooted sum of the family $\{\Tree_{(v\to)}\}_{v\in\Chif(u)}$ (cf.\ 
\eqref{eq:child-sum}).
It follows from the joint backward extension property that $S_{\blambda^\to_v}$
admits a $\Cls$ $(k+1)$-step backward extension for every $v\in\Chif(u)$
and every $u\in\Chif[k](\omega)$.
By Lemma~\ref{lem:tree-basics}~(\ref{tbs:dzieci}),
\[\Chif[k+1](\omega)=\bigcup_{u\in\Chif[k](\omega)}\Chif(u),\]
which shows that (b) is valid. This completes the proof.
\end{proof}

\begin{figure}[htb]
	\includegraphics[width=.9\textwidth]{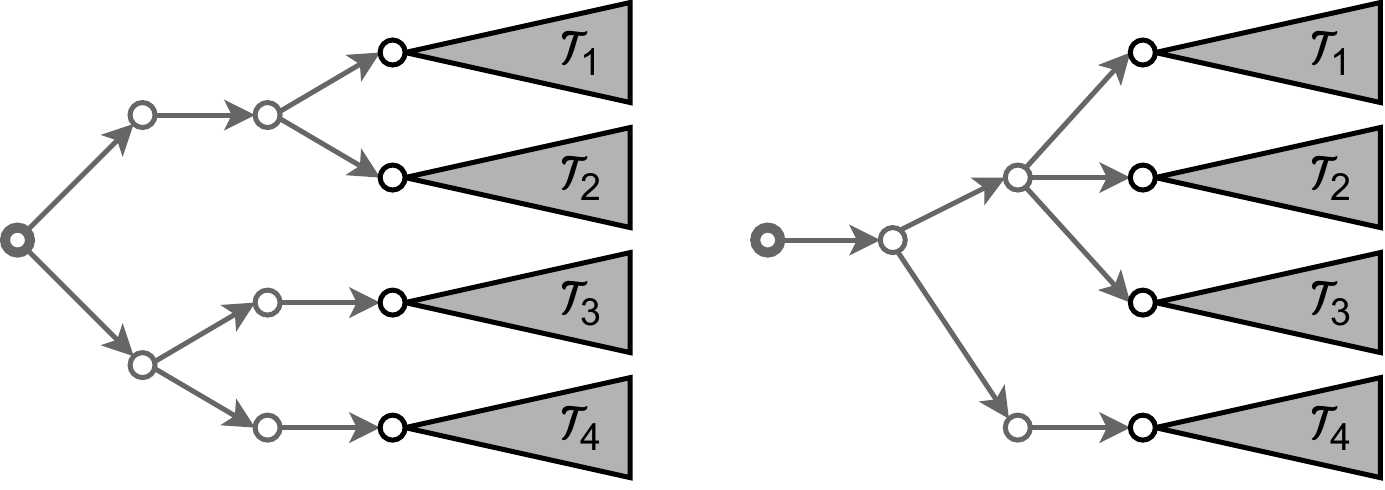}
	\caption[Two non-isomorphic joint backward extensions]{An example of two
		non-isomorphic directed trees completing a family
		$\{\Tree_1,\Tree_2,\Tree_3,\Tree_4\}$ at the depth $k=3$.
		According to Theorem~\ref{thm:joint-extension}, an extension in class $\Cls$
		exists either in both cases	or in neither.}
	\label{fig:two-trees}
\end{figure}

\begin{Rem}
Theorem~\ref{thm:joint-extension} can be interpreted in a different way.
Given a family of rooted directed trees we can complete them into
one directed tree by adding some new vertices such that the roots
of initial trees become the $k$-th children of the root of the new
tree. The value of this theorem is that the condition (b) does not
depend on the structure of the first $k$ generations of the enveloping
tree. It is worth pointing out that for $k\geq 2$ the same family of
trees can be completed in a variety of ways with enveloping directed
trees being not-isomorphic (see Figure~\ref{fig:two-trees}).
\end{Rem}

\begin{Rem}
From the proof of Theorem~\ref{thm:joint-extension} one can deduce,
that the construction of joint backward extension is possible also if
the roots of the trees from the predefined family are not necessarily
all on the same level. It suffices that if for some 
$v\in\Chif[j](\omega)$ the weighted shift $S_{\blambda^\to_v}$ is
fixed, then it admits $\Cls$ $j$-step backward extension.
\end{Rem}

\begin{Thm}\label{thm:sub-JBEP}
The class of bounded subnormal weighted shifts on directed trees
has the joint backward extension property.\end{Thm}
The above theorem follows from \cite[Lemma~5.6]{pikul}.
Theorem~5.7 therein was a particular case of
Theorem~\ref{thm:joint-extension} above.
In the forthcoming section we investigate whether classes of power
hyponormal and completely hyperexpansive weighted shifts satisfy JBEP.

\section{Power hyponormality}\label{sec:powhyp1}
All hyponormal classical weighted shifts on $\ell^2$ are automatically
power hyponormal. This is no longer the case for weighted shifts on
directed trees, as shown by \cite[Example~5.3.2.]{szifty}.
In \cite[Section~4]{pikul} a description of all those directed trees
on which every bounded hyponormal and proper weighted shift is power
hyponormal is given.

The characterisation of hyponormality of weighted shifts on directed
trees was presented in \cite[Theorem 5.1.2]{szifty}.
The theorem below is a generalisation to the case of arbitrary powers
of proper weighted shifts on directed forests.
It can be deduced from the result for trees or found in \cite[Section~4]{pikul}.

\begin{Thm}\label{thm:pow-hypo}
A bounded proper weighted shift $S_\blambda$ on a directed forest $\Tree=(V,\parf)$
is power hyponormal if and only if the following two conditions are satisfied:
\begin{abece}
\item the forest $\Tree$ is leafless,
\item for all $v\in V$ and $k\in \N$
\begin{equation}\label{eq:k-hypo}
\hip{k}(v)=\hip{k}_\blambda(v):=\sum_{u\in\Chif[k](v)}
\frac{|\lambda_u^{(k)}|^2}{\|S_\blambda^k e_u\|^2}\leq 1.
\end{equation}\end{abece}\end{Thm}


The following lemma reduces the question about admitting a power hyponormal
$k$-step backward extension for a weighted shift on a rooted directed tree to
the existence of a single number.
For the notion of the $k$-step backward extension 
of a directed tree see Definition~\ref{dfn:backward-ext}.

\begin{Lem} Fix $k\in\N$. Suppose that $\Tree\sqlow{k}=(V\sqlow{k},\parf\sqlow{k})$
is a $k$-step backward extension of a rooted directed tree $\Tree=(V,\parf)$ and
$S_\blambda$ is a bounded power hyponormal proper weighted shift on $\Tree\sqlow{k}$
with weights $\blambda=\{\lambda_v\}_{v\in V\sqlow{k}}$.
Let $\tlambda=\{\tilde\lambda_v\}_{v\in V\sqlow{k}}$ be a system of weights such
that $\tilde\lambda_v:=\lambda_v$ for $v\in V^\circ$, $\tilde\lambda_{\omega_j}:=\lambda_{\omega_{k-1}}$ for $j=0,\ldots,k-1$
and $\tilde\lambda_{\omega_k}:=0$.
Then the weighted shift $S_\tlambda$ on $\Tree\sqlow{k}$ is proper and power
hyponormal.
\end{Lem}
Note that both $S_\blambda$ and $S_\tlambda$ are power hyponormal
$k$-step backward extensions of the weighted shift $S_\blambda\rest{\eld{V}}$
on $\Tree$.
\begin{proof}
The leaflessness of $\Tree\sqlow{k}$ follows from the hyponormality of $S_\blambda$.
In the view of Theorem~\ref{thm:pow-hypo} all we have to prove is
that $\hip{n}_\tlambda(v)\leq 1$ for all $v\in V\sqlow{k}$ and $n\in\N$.
Since $\tilde\lambda_v=\lambda_v$ for $v\in V^\circ=\Deso(\omega)$, we see that
$\hip{n}_\tlambda(v)=\hip{n}_\blambda(v)\leq 1$ for $v\in \Des(\omega)$
and $n\in \N$.
It remains to estimate $\hip{n}_\tlambda(\omega_m)$
for $m=1,\ldots,k$ and $n\in\N$.

Set $\theta = |\tilde\lambda_{\omega_0}|^2=|\lambda_{\omega_{k-1}}|^2$.
Since
\[ 1\geq \hip{1}_\blambda(\omega_{j+2})=\frac{|\lambda_{\omega_{j+1}}|^2}
{|\lambda_{\omega_j}|^2},\quad j=0,\ldots,k-2,\] we conclude that
\begin{equation}\label{eq:theta-min}
|\tilde\lambda_{\omega_m}|^2=\theta\leq|\lambda_{\omega_m}|^2, \quad m=0,\ldots,k-1.
\end{equation}
If $2n\leq m$, then 
\[ \hip{n}_\tlambda(\omega_m)=
\frac{ |\tilde\lambda^{(n)}_{\omega_{m-n}}|^2 }{ |\tilde\lambda^{(n)}_{\omega_{m-2n}}|^2 }=\frac{\theta^n}{\theta^n}=1. \]
When $n\leq m < 2n$ we have
\[ \hip{n}_\tlambda(\omega_m)=
\frac{|\tilde\lambda^{(n)}_{\omega_{m-n}}|^2 }{\|S_\tlambda^n e_{\omega_{m-n}}\|^2 }=
\frac{ |\tilde\lambda^{(n)}_{\omega_{m-n}}|^2 }{ |\tilde\lambda^{(m-n)}_{\omega_0}|^2
\|S_\tlambda^{2n-m} e_{\omega_0}\|^2 }
=\frac{ \theta^{2n-m} }{ \|S_\tlambda^{2n-m} e_{\omega_0}\|^2 }.\]
Since $1\leq 2n-m\leq m$, it follows that
\[ 1\geq \hip{2n-m}_\blambda(\omega_{2n-m})=
\frac{ |\lambda^{(2n-m)}_{\omega_0}|^2 }{\|S_\blambda^{2n-m} e_{\omega_0}\|^2 }
\refoneq[\geq]{eq:theta-min}
\frac{ \theta^{2n-m} }{ \|S_\tlambda^{2n-m} e_{\omega_0}\|^2 }
=\hip{n}_\tlambda(\omega_m) .\]
We are left with the case when $m<n$, but then
\begin{align*}
\hip{n}_\tlambda(\omega_m)&\spacedeq{1ex}
\sum_{v\in\Chif[n](\omega_m)}
\frac{|\tilde\lambda^{(n)}_v|^2}{
	\|S_\tlambda^{n} e_v\|^2} 
=\sum_{v\in\Chif[n-m](\omega_0)}\!
\frac{ |\tilde\lambda^{(m)}_{\omega_0}|^2 |\tilde\lambda^{(n-m)}_v|^2 }%
{\|S_\tlambda^n e_v\|^2 }\\
&\refoneq[\leq]{eq:theta-min} \sum_{v\in\Chif[n-m](\omega_0)}\!\!
\frac{ |\lambda^{(m)}_{\omega_0}|^2 |\lambda^{(n-m)}_v|^2 }%
{\|S_\blambda^n e_v\|^2 }
=\hip{n}_\blambda(\omega_m)\leq1.
\end{align*}
This completes the proof of the power hyponormality of $S_\tlambda$.
\end{proof}

The following characterisation of weighted shifts admitting power hyponormal
backward extension holds.

\begin{Lem}\label{lem:powhyp-ext}
Let $S_\blambda$ be a nonzero bounded and power hyponormal proper weighted shift
on a directed tree $\Tree=(V,\parf)$ with a root $\omega$ and
$k$ be a positive integer.
Then \NWSR
\begin{abece}
\item $S_\blambda$ admits power hyponormal $k$-step backward extension,
\item there exists $C>0$ such that
\begin{equation}\label{eq:powhyp-kext}
\sum_{u\in\Chif[n](\omega)}\frac{|\lambda^{(n)}_u |^2}{
\|S_\blambda^{n+m} e_u\|^2} \leq C,\quad n\in\ZP,\ 1\leq m\leq k.
\end{equation}
\end{abece}
\end{Lem}
Recall that from properness and power hyponormality of $S_\blambda$ we
conclude that $\|S_\blambda^n e_u\|^2>0$ for all $n\in\ZP$ and $u\in V$.
\begin{proof}\pImp{a}{b}
Assume that the weighted shift $S_\tlambda$ on $\Tree\sqlow{k}$ is a power
hyponormal $k$-step backward extension of $S_\blambda$
(cf.\ Definition~\ref{dfn:backward-ext-shift}).

Denote $\theta = |\tilde\lambda_{\omega_0}|^2$. It is strictly positive,
since we do not accept zero weights for backward extensions.
In view of the previous lemma, we can assume that
$\tilde\lambda_{\omega_m}=\tilde\lambda_{\omega_0}$ for $m=1,\ldots,k-1$.

Since $S_\tlambda$ is power hyponormal, by Theorem~\ref{thm:pow-hypo} we have
\begin{align*}
1\geq \hip{n+m}_\tlambda(\omega_m)=
\sum_{u\in\Chif[n+m](\omega_m)}
\frac{|\tilde\lambda^{(n+m)}_u|^2}{
	\|S_\tlambda^{n+m} e_u\|^2} &
=\sum_{u\in\Chif[n](\omega_0)}
\frac{ \theta^m |\tilde\lambda^{(n)}_u|^2 }{
\|S_\tlambda^{n+m} e_u\|^2 }
\end{align*}
for $n\in\ZP$ and $1\leq m\leq k$. Consequently
\[ \sum_{u\in\Chif[n](\omega_0)}
\frac{ |\tilde\lambda^{(n)}_u|^2 }{
\|S_\tlambda^{n+m} e_u\|^2 } \leq \frac{1}{\theta^m},\quad
n\in\ZP,\ 1\leq m\leq k.\]
Hence, setting $C=\max\{1, 1/\theta^k\}$
is sufficient for (b).

\pImp{b}{a}
We can find a constant $\theta\in (0,1)$ such that
\begin{equation} \label{eq:theta-n}
\frac{1}{\theta^m} \geq \sum_{u\in\Chif[n](\omega)}
\frac{|\lambda^{(n)}_u|^2}{\|S_\blambda^{n+m} e_u\|^2},
\quad n\in\ZP,\ 1\leq m\leq k
\end{equation} and also
\begin{equation} \label{eq:theta-m}
\|S_\blambda^m e_\omega\|^2 \geq \theta^m, \quad 1\leq m\leq k.
\end{equation}

Using this constant we set $\tilde\lambda_{\omega_m}:=\sqrt{\theta}$
for $m=0,\ldots,k-1$, $\tilde\lambda_{\omega_k}:=0$ and
$\tilde\lambda_v:=\lambda_v$ for $v\in V^\circ$. We claim
that the weighted shift $S_\tlambda$ with weights
$\tlambda=\{\tilde\lambda_v\}_{v\in V\sqlow{k}}$ is power
hyponormal.

Clearly, $\hip{n}_\tlambda (v)=\hip{n}_\blambda (v)\leq 1$ for all $v\in V$
and $n\in\N$. Fix $1\leq m\leq k$.
If $2n\leq m$, then
$\hip{n}_\tlambda (\omega_m) = \frac{\theta^n}{\theta^n}=1$.

In the case $n\leq m <2n$ we have
\begin{align*}
\hip{n}_\tlambda(\omega_m)=
\frac{| \tilde\lambda^{(n)}_{\omega_{m-n}} |^2}{\|S_\tlambda^n e_{\omega_{m-n}}\|^2 }
&=\frac{ \theta^n }{\theta^{m-n} \|S_\tlambda^{2n-m} e_{\omega_0}\|^2 }\\
&=\frac{ \theta^{2n-m} }{\|S_\tlambda^{2n-m} e_{\omega_0}\|^2 }
\refoneq[\leq]{eq:theta-m} 1.
\end{align*}
We are left with the case $m<n$. Then
\begin{align*}
\hip{n}_\tlambda(\omega_m) &=
\sum_{u\in\Chif[n](\omega_m)}
\frac{ |\tilde\lambda^{(n)}_u |^2}{\|S_\tlambda^n e_u\|^2 }\\
&=\sum_{u\in\Chif[n-m](\omega_0)}
\frac{\theta^m |\tilde\lambda^{(n-m)}_u|^2}{
\|S_\tlambda^n e_u\|^2 }
\refoneq[\leq]{eq:theta-n} 1.
\end{align*}
Since, by Theorem~\ref{thm:pow-hypo}, $\Tree$ is leafless and by
$S_\blambda\neq 0$ it is non-degenerate, $\Tree\sqlow{k}$ is leafless.
By the same theorem the weighted shift $S_\tlambda$ on $\Tree\sqlow{k}$
is a power hyponormal $k$-step backward extension of $S_\blambda$ so (a) is proven.
\end{proof}

From the above characterisation we can deduce the following fact
about backward extensions of power hyponormal weighted shifts.

\begin{Cor}\label{cor:powhypo-below}
If a bounded power hyponormal proper weighted shift on a rooted directed tree
is bounded from below, then it admits a power hyponormal $k$-step backward
extension for every $k\in\N$.
\end{Cor}
\begin{proof}
Denote the considered weighted shift on a rooted directed tree $\Tree=(V,\parf)$
by $S_\blambda$.
Let $C\in(0,1]$ be such that
\[ \|S_\blambda f\|^2 \geq C\|f\|^2,\quad
f\in\eld{V}.\]
If $n\in \N$ and $m=1,\ldots,k$, then by Theorem~\ref{thm:pow-hypo},
\[
\sum_{u\in\Chif[n](\omega)}\frac{|\lambda^{(n)}_u|^2}{
\|S_\blambda^{n+m} e_u\|^2}\leq 
\sum_{u\in\Chif[n](\omega)}\frac{|\lambda^{(n)}_u|^2}{
\|S_\blambda^{n} e_u\|^2 C^m}\leq \frac{\hip{n}(\omega)}{C^m} \leq C^{-k}.
\]
For $n=0$, we obtain
\[ \sum_{u\in\Chif[0](\omega)}\frac{|\lambda^{(0)}_u|^2}{
\|S_\blambda^{m} e_u\|^2}=
\frac{1}{\|S_\blambda^{m} e_\omega\|^2 }\leq \frac{1}{C^m} \leq C^{-k}. \]
Applying Lemma~\ref{lem:powhyp-ext} finishes the proof.
\end{proof}

Now we are ready to prove that the class of power hyponormal weighted shifts
satisfies the joint backward extension property (with $c_k=1$).
For the reader's convenience, statement of the theorem contains details
of the JBEP. See Definition~\ref{dfn:rooted-sum}
for the notion of rooted sum of directed trees.

\begin{Thm}\label{thm:powhyp-JBEP}
Let $J$ be a non-empty at most countable set and $k\in\ZP$.
Suppose that for $j\in J$, $S_j$ is a bounded proper weighted shift on a
directed tree $\Tree_j=(V_j,\parf_j)$ with root $\omega_j$. Then \NWSR
\begin{abece}
\item there exists a system $\{ \theta_j \}_{j\in J}\subseteq \C\setminus\{0\}$
such that the weighted shift $S_\blambda$ on
$\rootsum_{j\in J}\Tree_j=(V_J,\parf)$
with weights $\blambda =\{\lambda_v\}_{v\in V_J}$ satisfying
\begin{equation}\label{eq:weight-rest1}
\lambda_{\omega_j}=\theta_j \text{ and } S_{\blambda^\to_{\omega_j}} = S_j
\text{ for all } j\in J,
\end{equation}
is bounded and admits a power hyponormal $k$-step backward extension,

\item $S_j$ admits a power hyponormal $(k+1)$-step backward extension for
each $j\in J$ and $\sup\{\|S_j\|\colon j\in J\} < \infty$.
\end{abece}
Moreover, if $S_\blambda$ is a bounded proper weighted shift on
$\rootsum_{j\in J}\Tree_j$ that admits a power hyponormal $k$-step
backward extension, then $\|S_\blambda\|=\sup\{\|S_j\|\colon j\in J\}$.
\end{Thm}
\begin{proof} Set $\Tree=\rootsum_{j\in J}\Tree_j$.
\pImp{a}{b}
The uniform boundedness part of (b) is obvious.
According to Lemma~\ref{lem:powhyp-ext} we have to prove that for every
$j\in J$, there exists a constant $C_j>0$ such that
\[ \sum_{u\in\Chit{n}{\Tree_j}(\omega_j)}\frac{\|S_j^{*n} e_u\|^2}{
\|S_j^{n+m} e_u\|^2} \leq C_j,\quad n\in\ZP,\ 1\leq m\leq k+1.\]

Fix $j_0\in J$. Note that assuming $C\geq 1$, the inequality \eqref{eq:powhyp-kext}
is satisfied also for $m=0$, since $\hip{n}(\omega)\leq 1\leq C$.
By the assumption on $S_\blambda$ and Lemma~\ref{lem:powhyp-ext} we have
\begin{align*}
C&\geq \sum_{u\in\Chit{n+1}{\Tree}(\omega)}\frac{|\lambda^{(n+1)}_u|^2}{
\|S_\blambda^{n+1+(m-1)} e_u\|^2} \\
&=\sum_{j\in J} \sum_{u\in\Chit{n}{\Tree_j}(\omega_j)}
\frac{|\theta_j|^2 |\lambda^{(n)}_u|^2}{
\|S_\blambda^{n+m} e_u\|^2}\\
&\geq |\theta_{j_0}|^2\!\! \sum_{u\in\Chit{n}{\Tree_{j_0}}(\omega_{j_0})}
\frac{|\lambda^{(n)}_u|^2}{\|S_\blambda^{n+m} e_u\|^2},\quad
1\leq m\leq k+1,\ n\in \ZP.
\end{align*}
This leads to
\[ \sum_{u\in\Chit{n}{\Tree_{j_0}}(\omega_{j_0})}\frac{\|S_{j_0}^{*n} e_u\|^2}{
\|S_{j_0}^{n+m} e_u\|^2} =\sum_{u\in\Chit{n}{\Tree_{j_0}}(\omega_{j_0})}
\frac{|\lambda^{(n)}_u|^2}{\|S_\blambda^{n+m} e_u\|^2}
\leq \frac{C}{|\theta_{j_0}|^2} =:C_{j_0} \]
for $m=1,\ldots,k+1$ and $n\in \ZP$, which completes the proof of (b).

\pImp{b}{a}
Since operators in question cannot be zero, we may and do assume that
$\sup_{j\in J}\|S_j\| = 1$.

It follows from Lemma~\ref{lem:powhyp-ext} that for every $j\in J$ there exists
a constant $C_j\geq 1$ such that
\begin{equation}\label{eq:Sj-hskos}
\sum_{u\in\Chit{n}{\Tree_j}(\omega_j)}\frac{\|S_j^{*n} e_u\|^2}{
\|S_j^{n+m} e_u\|^2} \leq C_j,\quad n\in\ZP,\ 1\leq m\leq k+1.
\end{equation} 
Pick a system of positive real numbers $\{a_j\}_{j\in J}$ such that
\[
\mathrm{(*)}\ \sum_{j\in J} a_j C_j^3 <\infty,
\qquad \mathrm{(**)}\  \sum_{j\in J} a_j C_j=1. \]
Indeed, as $J$ is at most countable, there exist positive real numbers
$\{a'_j\}_{j\in J}$ such that $\sum_{j\in J} a_j'=1$. Then, by setting $a_j:=a_j'\min\{1,C_j^{-3}\}$ for $j\in J\setminus\{j_0\}$,
where $j_0\in J$ is fixed, and defining \[ a_{j_0}:=
\biggl(1-\sum_{j\in J\setminus\{j_0\}} a_j C_j\biggr)/C_{j_0}, \]
we obtain ($*$) and ($**$).

Take $\theta_j$'s such that $|\theta_j|^2=a_j$ for all $j\in J$ and
define $\blambda =\{\lambda_v\}_{v\in V_J}$ by \eqref{eq:weight-rest1}.
We claim that this system meets the condition (a). Clearly
\[ \sum_{u\in\Chif(\omega)}|\lambda_u|^2=\sum_{j\in J} a_j
\leq \sum_{j\in J} a_j C_j \stackrel{(**)}{=}1.\]
According to Proposition~\ref{pro:shift-basic}~\eqref{shbs:bounded}, this
together with uniform boundedness of $S_j$'s proves that $S_\blambda$ is bounded.

Let $n$ be a positive integer. Then
\begin{align*}
\hip{n}_\blambda(\omega) &\spacedeq{1ex} 
\sum_{u\in\Chit{n}{\Tree}(\omega)}\frac{\|S_\blambda^{*n} e_u\|^2}{\|S_\blambda^n e_u\|^2}\\
&\spacedeq{1ex} \sum_{j\in J} a_j \sum_{u\in\Chit{n-1}{\Tree_j}(\omega_j)}\frac{\|S_\blambda^{*n-1} e_u\|^2}{\|S_\blambda^n e_u\|^2}\\
&\refoneq[\leq]{eq:Sj-hskos} \sum_{j\in J} a_j C_j = 1.
\end{align*}
Since all $S_j$'s are power hyponormal, we infer from Theorem~\ref{thm:pow-hypo}
that $S_\blambda$ is power hyponormal. If $k=0$, then we are done.

In the case $k\geq 1$, according to Lemma~\ref{lem:powhyp-ext}
it suffices to prove that there exists a constant $C>0$ satisfying 
\eqref{eq:powhyp-kext}.
We will consider two possible cases.

Case 1. $n=0$.

First, observe that for $j\in J$,
\begin{equation}\label{eq:Cj/norma}
 \frac{C_j}{\|S_j^{m} e_{\omega_j}\|^2}=
\sum_{u\in\Chif[0](\omega_j)}\frac{C_j\|S_j^{*0} e_u\|^2}{\|S_j^{m} e_u\|^2}\refoneq[\leq]{eq:Sj-hskos} C_j^2,\quad m=0,\ldots,k.
\end{equation}
Hence, for $m\in\{1,\ldots,k\}$,
\begin{align*}
\sum_{u\in\Chif[0](\omega)}
\frac{\|S_\blambda^{*0} e_u\|^2}{\|S_\blambda^{m} e_u\|^2}
=\frac{1}{\|S_\blambda^m e_\omega\|^2}
&\spacedeq{.9ex} \frac{1}{\sum_{j\in J} a_j \|S_j^{m-1} e_{\omega_j}\|^2}\\
= \frac{1}{\sum_{j\in J} a_j C_j \frac{\|S_j^{m-1} e_{\omega_j}\|^2}{C_j}}
&\spacedeq[\stackrel{(\dagger)}{\leq}]{.9ex} \sum_{j\in J} a_j C_j \frac{C_j}{\|S_j^{m-1} e_{\omega_j}\|^2}\\
&\refoneq[\leq]{eq:Cj/norma} \sum_{j\in J} a_j C_j^3 \stackrel{(*)}{<} \infty,
\end{align*}
where ($\dagger$) follows from ($**$) and the Jensen inequality applied to
the convex function $x\mapsto\frac{1}{x}$.

Case 2. $n\in \N$.

For $m\in\{1,\ldots,k\}$, we have
\begin{align*}
\sum_{u\in\Chit{n}{\Tree}(\omega)}\frac{\|S_\blambda^{*n} e_u\|^2}{\|S_\blambda^{n+m} e_u\|^2}
&\spacedeq{.9ex} \sum_{j\in J} \sum_{u\in\Chit{n-1}{\Tree}(\omega_j)}\frac{\|S_\blambda^{*n} e_u\|^2}{\|S_\blambda^{n+m} e_u\|^2}\\
&\spacedeq{.9ex} \sum_{j\in J} a_j \sum_{u\in\Chit{n-1}{\Tree_j}(\omega_j)}
\frac{\|S_j^{*(n-1)} e_u\|^2}{\|S_j^{n+m} e_u\|^2}\\
&\refoneq[\leq]{eq:Cj/norma} \sum_{j\in J} a_j C_j \stackrel{(**)}{=} 1.
\end{align*}
Together with the previous estimation we obtain that the condition (b) in
Lemma~\ref{lem:powhyp-ext} is satisfied with constant
$C:=\sum_{j\in J} a_j C_j^3 \geq 1$.
Hence $S_\blambda$ admits a power hyponormal $k$-step backward extension.

For the ``moreover'' part observe that \[
\|S_\blambda^* (S_\blambda e_w) \|^2
= \Bigl\| \sum_{j\in J} \lambda_{\omega_j} S_\blambda^* e_{\omega_j}\Bigr \|^2
\refoneq{eq:S*dok} \Bigl\| \sum_{j\in J} |\lambda_{\omega_j}|^2 e_w \Bigr \|^2
= \|S_\blambda e_w\|^4,
\]
and
\begin{align*}
 \|S_\blambda (S_\blambda e_w) \|^2
&\refoneq{eq:normasumy} \sum_{j\in J} |\lambda_{\omega_j}|^2
\| S_\blambda e_{\omega_j} \|^2
\leq \sum_{j\in J} |\lambda_{\omega_j}|^2 \| S_j \|^2\\
&\spacedeq[\leq]{1ex} \| S_\blambda e_w \|^2 \sup \bigl\{\|S_j\|^2\colon
j\in J\bigr\}.
\end{align*}
By hyponormality we have $\|S_\blambda^* (S_\blambda e_w) \|^2
\leq \|S_\blambda (S_\blambda e_w) \|^2$, and consequently
$\|S_\blambda\|\leq \sup\{\|S_j\|\colon j\in J\}$.
\end{proof}

As stated in Theorem~\ref{thm:joint-extension}, the joint backward extension
property applies not only to extensions to weighted shifts onto rooted sums.

\begin{Rem}\label{rem:blank-powehypo}
According to \cite[Proposition~3.9]{pikul}, if some
countable leafless directed subtree $\Tree_{(v\to)}$ is given without predefined
weights, we can set them in such a way that the associated weighted shift
admits isometric (subnormal, power hyponormal) $k$-step backward extension,
and hence such ``blank directed tree'' does not affect the existence of
a subnormal (power hyponormal) extension.
\end{Rem}

\section{Complete hyperexpansivity}\label{sec:che1}
Given an operator $T\in\BO{\Hilb}$ we define
\begin{equation}\label{eq:hyperexA}
A_n(T):=\sum_{j=0}^n (-1)^j \binom{n}{j} T^{*j}T^j,\quad n\geq 0.
\end{equation}
One can observe that $A_{n+1}(T)=A_n(T) - T^* A_n(T)T $ for every $n\in\ZP$ and $A_0(T)=I$.
We call $T\in\BO{\Hilb}$ \emph{completely hyperexpansive} if
$A_n(T)\leq 0$ for all $n\in\N$. The notion was introduced by Athavale in \cite{athav}.
Due to the following result of Agler \cite{agler} the complete hyperexpansivity
is in a sense `dual' to subnormality (and contractivity).
\begin{Thm}[Agler]
An operator $T\in\BO{\Hilb}$ is a subnormal contraction if and only if
$A_n(T)\geq 0$ for every $n\in\N$.
\end{Thm}

From the aforementioned recurrence easily follows that every isometry is
a completely hyperexpansive operator, such that
\[ A_n(T)= I-T^*T =0, \quad n\geq 1.\]

It is worth noting that for a completely hyperexpansive operator $T\in\BO{\Hilb}$
also $T^n$ is completely hyperexpansive for all $n\in\N$. Proof of this fact can be found
in \cite[Theorem 2.3]{jab}.

Hyperexpansivity is closely related to the following notion.
\begin{Def}
A sequence $\{a_n\}_{n=0}^\infty$ is called \emph{completely alternating} if
\[ \sum_{j=0}^n (-1)^j\binom{n}{j} a_{m+j} \leq 0,\quad n\in\N,\ m\in\ZP. \]
\end{Def}

Some basic properties of these sequences are listed below. Their elementary
proofs are omitted.
\begin{Pro}\label{pro:alt-seq}
Let $\{a_n\}_{n=0}^\infty$ be a sequence of real numbers. Define
\begin{equation}\label{eq:Amn}
A_m^n:=\sum_{j=0}^n (-1)^j\binom{n}{j} a_{m+j},\qquad m,n\in\ZP.
\end{equation} Then
\begin{abece}
\item $A_m^0 = a_m$ and $A_m^{n+1} = A_m^n - A_{m+1}^n$ for $m,n\in\ZP$,
\item if $\{a_n\}_{n=0}^\infty$ is completely alternating, then so is
$\{a_{n+1}\}_{n=0}^\infty$,
\item if $\{a_n\}_{n=0}^\infty$ is completely alternating, then
the sequence $\{A_m^n\}_{m=0}^\infty$ is monotonically increasing for every $n\in\ZP$
$($in particular $\{a_n\}_{n=0}^\infty$ is monotonically increasing$)$,
\item if $\{a_n\}_{n=0}^\infty$ is completely alternating sequence with positive terms,
then quotients of consecutive terms form a monotonically decreasing
sequence, i.e.\ $\frac{a_{n+1}}{a_n}\geq\frac{a_{n+2}}{a_{n+1}}$ for $n\geq 0$,
\item completely alternating sequences form a convex cone in $\R^\ZP$
closed with respect to pointwise convergence.
\end{abece}
\end{Pro}

More facts about completely alternating sequences can be found in e.g.\ 
\cite{berg}. 

The relation between completely hyperexpansive operators and completely alternating
sequences is presented in the following lemma. In \cite{szifty} the criterion
appearing in the lemma is used as a definition of complete hyperexpansivity.

\begin{Lem}\label{lem:che-general}
An operator $T\in\BO{\Hilb}$ is completely hyperexpansive if and only if
the sequence $\{\|T^n f\|^2\}_{n=0}^\infty$ is completely alternating
for every $f\in\Hilb$.
\end{Lem}
\begin{proof}
It is a consequence of \eqref{eq:hyperexA}, \eqref{eq:Amn} and
the definition of a negative definite operator.
\end{proof}

Completely alternating sequences are somehow similar to the moment sequences
mentioned in the section about subnormality. The following proposition
provides a handy measure-theoretic characterisation.
Its proof can be found in \cite[Proposition 4.6.12]{berg}.
\begin{Thm}\label{thm:compalt-measure}
A sequence $\{a_n\}_{n=0}^\infty \subseteq \R$ is completely alternating
if and only if
there exists a positive Borel measure $\tau$ on $[0,1]$ such that
\begin{equation}\label{eq:altmeasure0}
a_n = a_0 + \int_{[0,1]}\sum_{j=0}^{n-1} t^j \ud \tau(t),\quad n\in\N.
\end{equation}
\end{Thm}
\begin{Rem}
The formulation of \cite[Proposition 4.6.12]{berg} mentions a Ra\-don measure
$\mu$ on $[0,1)$ such that \[a_n = a+ nb +\int_{[0,1)}(1-x^n)\ud\mu(x),
\quad n\in\ZP,\] for some $a,b\in\R$ with $b\geq 0$.
Measure $\tau$ from \eqref{eq:altmeasure0} is then given by the following
one-to-one correspondence \[\tau(A)=b\delta_1(A)+\int_{A\setminus\{1\}}(1-x)\ud\mu(x),
\quad A\in\Borel([0,1]).\]
A more detailed comment can be found in \cite[Remark~1]{athav}.
\end{Rem}

\begin{Cor}\label{cor:alt-unique}
Measure $\tau$ appearing in Theorem~\ref{thm:compalt-measure} is unique
for a completely alternating sequence $\{a_n\}_{n=0}^\infty$.
\end{Cor}
\begin{proof}
Consider the values $A_m^1$ defined by \eqref{eq:Amn}. Using the recursive formula
(see Proposition~\ref{pro:alt-seq}) we conclude from \eqref{eq:altmeasure0} that
\[ A_m^1=-\int_{[0,1]} t^m\ud\tau(t),\quad m\in\ZP.\]
As a consequence we obtain that given two Borel measures $\tau$ and $\tau'$
satisfying \eqref{eq:altmeasure0} we have
$\int_{[0,1]} p(t)\ud\tau(t)=\int_{[0,1]} p(t)\ud\tau'(t)$ for any polynomial $p$,
and hence $\tau=\tau'$ by uniform density of polynomials among continuous functions
on $[0,1]$.
\end{proof}

Measure $\tau$ satisfying \eqref{eq:altmeasure0} will be called
\emph{the representing measure} of the completely alternating sequence
$\{a_n\}_{n=0}^\infty$. 
It can be easily observed that the map assigning representing measures to
completely alternating sequences is ``linear''.

Using \eqref{eq:normasumy}, Lemma~\ref{lem:che-general} and
Proposition~\ref{pro:alt-seq}~(e) we obtain the following characterisation
(we use the notation $\int_0^1 = \int_{[0,1]}$):
\begin{Cor}\label{cor:shift-hypexp}
Let $S_\blambda$ be a bounded weighted shift on a directed forest $\Tree=(V,\parf)$.
Then \NWSR\begin{abece}
\item $S_\blambda$ is completely hyperexpansive,
\item the sequence $\{ \|S_\blambda^n e_v\|^2 \}_{n=0}^\infty$ is
completely alternating for every $v\in V$,
\item for every $v\in V$, there exists a $($unique$)$ Borel measure $\tau_v$ on $[0,1]$
such that $\|S_\blambda^{n+1} e_v\|^2-\|S_\blambda^n e_v\|^2
=\int_0^1 t^n \ud\tau_v(t)$, for $n\in\ZP$.
\end{abece}
Moreover,  $\tau_v$ from {\rm (c)} is the representing measure for
the completely alternating sequence $\{ \|S_\blambda^n e_v\|^2 \}_{n=0}^\infty$.
\end{Cor}

By Proposition~\ref{pro:alt-seq}~(c), $\ker T=\{0\}$, provided that
$T\in\BO{\Hilb}$ is completely hyperexpansive. Consequently, if $S_\blambda$ is
a~completely hyperexpansive weighted shift on a directed forest $\Tree$,
then $\Tree$ contains neither leaves nor degenerate trees
(i.e.\ $\Chif(v)\neq\emptyset$ for every vertex $v$ of $\Tree$).

Knowing the characterisation given in Corollary~\ref{cor:shift-hypexp},
a natural starting point for the study of backward extensions of completely
hyperexpansive weighted shifts is the question about extensions of completely
alternating sequences. The following lemma addresses this problem.
It generalises \cite[Lemma 7.1.2]{szifty}.
\begin{Lem}\label{lem:compalt-ext}
Let $\{a_n\}_{n=0}^\infty$ be a completely alternating sequence with the
representing measure $\tau$ and $k$ be a positive integer. Then \NWSR
\begin{abece}
\item there exist $\{a_{-j}\}_{j=1}^k \subseteq\R$ such that the
sequence $\{a_{n-k}\}_{n=0}^\infty$ is completely alternating
with $a_{-k}=1$,
\item $\tau(\{0\})=0$ and
\begin{equation}\label{eq:altext}
\int_0^1\left( \frac{1}{t}+\ldots+\frac{1}{t^k}\right)\ud\tau(t)\leq a_0-1.
\end{equation}
\end{abece}
Moreover, if {\rm (a)} is satisfied, then the Borel measure $\rho$ defined by
\begin{align}\label{eq:roztau}
\rho(A):=\int_A \frac{1}{t^k}\ud\tau(t)&+ \left(a_0-1-
\int_0^1\Bigl( \frac{1}{t}+\ldots+\frac{1}{t^k}\Bigr)\ud\tau(t)
\right)\delta_0(A),\nonumber \\
&\hspace{26ex} A\in\Borel([0,1]),
\end{align}
is representing for $\{a_{n-k}\}_{n=0}^\infty$ and $a_n\geq 1$ for all $n\geq -k$.
\end{Lem}
\begin{proof}
\pImp{a}{b}
Let $\rho_0$ be the representing measure of a completely alternating sequence
$\{a_{n-k}\}_{n=0}^\infty$ (see Theorem~\ref{thm:compalt-measure}). Define the Borel
measure $\tau_0$ on $[0,1]$ by the formula $\tau_0(A):=\int_A t^k \rho_0(t)$
for $A\in\Borel([0,1])$.
Then
\begin{align*}
a_n &= a_{-k}+\int_0^1 \Bigl(1+\ldots+t^{n+k-1}\Bigr)\ud\rho_0(t) \\
 &= a_{-k}+\int_0^1 \Bigl(1+\ldots+t^{k-1}\Bigr)\ud\rho_0(t)
+\int_0^1 \Bigl(t^k+\ldots+t^{n+k-1}\Bigr)\ud\rho_0(t) \\
&= a_0 + \int_0^1 \Bigl( 1+\ldots+t^{n-1}\Bigr) t^k \ud\rho_0(t) \\
&= a_0 + \int_0^1 \Bigl(1+\ldots+t^{n-1}\Bigr) \ud\tau_0(t),\quad n\in\N,
\end{align*}
and by the uniqueness of representing measure, $\tau_0=\tau$ and
consequently $\tau(\{0\})=\tau_0(\{0\})=0$.

Using \eqref{eq:altmeasure0} to $\rho_0$ and $\{a_{n-k}\}_{n=0}^\infty$
with $n=k$ we obtain
\begin{align*}
a_0 &= a_{-k}+\int_0^1\Bigl( 1+\ldots+t^{k-1} \Bigr)\ud\rho_0(t) \\
&\geq 1 + \int_{(0,1]}\Bigl( 1+\ldots+t^{k-1} \Bigr)\ud\rho_0(t) \\
&=1+\int_{(0,1]}\Bigl( \frac{1}{t^k}+\ldots +\frac{1}{t} \Bigr)\ud\tau(t)\\
&=1+\int_0^1\Bigl( \frac{1}{t^k}+\ldots +\frac{1}{t} \Bigr)\ud\tau(t).
\end{align*}
This proves the inequality \eqref{eq:altext}.

For the `moreover' part we need to verify that $\rho$ given by \eqref{eq:roztau}
is a representing measure for $\{a_{n-k}\}_{n=0}^\infty$, i.e.\ $\rho=\rho_0$.
If $A\in\Borel([0,1])$ and $0\notin A$, then
\[\rho(A) \refoneq{eq:roztau} \int_A\frac{1}{t^k}\ud\tau(t)=
\int_A\frac{1}{t^k}\ud\tau_0(t)=\int_A\frac{t^k}{t^k}\ud\rho_0(t)=\rho_0(A).\]
To complete the proof of $\rho=\rho_0$ it suffices to compute $\rho_0(\{0\})$.
Arguing as in the previous paragraph, we get
\begin{align*}
a_0 -1&= a_{-k}-1 +\int_0^1 \Bigl(1+\ldots+t^{k-1} \Bigr)\ud\rho_0(t) \\
&= \rho_0(\{0\})+ \int_{(0,1]} \Bigl(\frac{1}{t^k}+\ldots
+\frac{1}{t}\Bigr) t^k\ud\rho_0(t)\\
&= \rho_0(\{0\})+ \int_0^1 \Bigl(\frac{1}{t^k}+\ldots
+\frac{1}{t} \Bigr)\ud\tau(t).
\end{align*}
Hence
\[\rho_0(\{0\})=a_0 -1-\int_0^1 \Bigl( \frac{1}{t^k}+\ldots
+\frac{1}{t} \Bigr)\ud\tau(t) \refoneq{eq:roztau} \rho(\{0\}). \]

That $\{a_{n-k}\}_{n=0}^\infty\subseteq [1,\infty)$ follows from
Proposition~\ref{pro:alt-seq}~(c) and $a_{-k}=1$.

\pImp{b}{a}
Define the Borel measure $\rho$ by \eqref{eq:roztau}. By \eqref{eq:altext}
it is a well defined positive measure. Put $a_{-k}:=1$ and
\begin{equation}\label{eq:a-j}
a_{-j}:=a_{-k}+\int_0^1 \Bigl( 1+\ldots+t^{j-1} \Bigr) \ud\rho(t),
\quad j=1,\ldots,k-1.
\end{equation}

According to Theorem~\ref{thm:compalt-measure} it suffices to prove that
$\rho$ satisfies \eqref{eq:altmeasure0} for the sequence $\{a_{n-k}\}_{n=0}^\infty$.
With convention that $\sum_{j=N}^{N-1} f(j) = 0$, we obtain
\begin{align*}
a_{-k} +\int_0^1 &\sum_{j=0}^{n-1} t^j \ud\rho(t)\\
&\spacedeq{1.3ex} 1 + \rho(\{0\}) + \int_{(0,1]} \sum_{j=0}^{n-1}t^j \ud\rho(t)\\
&\spacedeq{1.3ex} 1+ \rho(\{0\}) + \int_{(0,1]} \sum_{j=0}^{k-1} \frac{1}{t^{j+1}} \ud\tau(t)
+\int_{(0,1]} \sum_{j=k}^{n-1} t^{j-k} \ud\tau(t)\\
&\refoneq{eq:roztau} a_0+ \int_0^1 \sum_{j=0}^{n-k-1} t^j \ud\tau(t) = a_{n-k},\quad n\geq k.
\end{align*}
This, combined with \eqref{eq:a-j}, completes the proof.
\end{proof}

Having the above lemma we can pass to a characterisation of those weighted shifts
that admit completely hyperexpansive $k$-step backward extensions. This is a
generalisation of \cite[Theorem 4.2.($1^\circ$)]{hypexp-ext} to the case of weighted
shifts on directed trees.

\begin{Thm}\label{thm:che-ext} 
For a bounded completely hyperexpansive weighted shift $S_\blambda$ on a
directed tree $\Tree=(V,\parf)$ with the root $\omega$ and a number $k\geq 1$, \NWSR\begin{abece}
\item $S_\blambda$ admits a completely hyperexpansive $k$-step backward extension,
\item $\int_0^1 \bigl( \frac{1}{t}+\ldots+\frac{1}{t^k} \bigr)\ud\tau(t) < 1$,
where $\tau$ is the representing measure for the completely alternating sequence
$\{\|S_\blambda^n e_\omega\|^2\}_{n=0}^\infty$.
\end{abece}\end{Thm}
\begin{proof}
Clearly, $S_\blambda\neq0$ because $A_1(0)=I \nleq 0$ (see \eqref{eq:hyperexA}).
\pImp{a}{b}
Let $S_{\blambda '}$ be a completely hyperexpansive $k$-step backward extension
of $S_\blambda$, i.e.\ a weighted shift on $\Tree\sqlow{k}$. Denote the root of
$\Tree\sqlow{k}$ by $\omega_k$ (cf.\ Definition~\ref{dfn:backward-ext}).
By the complete hyperexpansivity of $S_{\blambda '}$ and Lemma~\ref{lem:che-general}
the sequence $\{\|S_{\blambda '}^n e_{\omega_k}\|^2\}_{n=0}^\infty$
is a completely alternating sequence with the representing measure, say $\tau_k$.

Since $S_{\blambda '}^k e_{\omega_k} = \lambda'^{(k)}_\omega e_\omega \neq 0$,
we have 
\begin{equation}\label{eq:Somegak}
\|S_\blambda^n e_\omega\|^2 = \frac{\|S_{\blambda '}^{n+k}
 e_{\omega_k}\|^2}{\bigl|\lambda'^{(k)}_\omega\bigr|^2},\quad n\geq 0.
\end{equation}
Then the Borel measure $\tau$ given by 
\begin{equation}\label{eq:def-tau}
\tau(A):=\int_A t^k\,|\lambda'^{(k)}_\omega|^{-2}\ud\tau_k(t),
\quad A\in\Borel([0,1]), \end{equation}
is a representing measure for $\{\|S_\blambda^n e_\omega\|^2\}_{n=0}^\infty$ such
that $\tau(\{0\})=0$. Indeed, according to Corollary~\ref{cor:shift-hypexp}~(c),
\begin{align*}
\|S_\blambda^{n+1} e_\omega\|^2-\|S_\blambda^n e_\omega\|^2 
&\refoneq{eq:Somegak} \left( \|S_{\blambda'}^{n+k+1} e_{\omega_k}\|^2
-\|S_{\blambda'}^{n+k} e_{\omega_k}\|^2\right) \bigl|\lambda'^{(k)}_\omega\bigr|^{-2}\\
&\spacedeq{1.2ex} \bigl|\lambda'^{(k)}_\omega\bigr|^{-2} \int_0^1 t^{n+k} \ud\tau_k(t)\\
&\spacedeq{1.2ex} \int_0^1 t^n t^k \bigl|\lambda'^{(k)}_\omega\bigr|^{-2} \ud\tau_k(t)
\refoneq{eq:def-tau} \int_0^1 t^n \ud\tau(t)
\end{align*}
for $n\in\ZP$.
Further we have
\begin{align*}
\bigl|\lambda'^{(k)}_\omega\bigr|^2 = \|S_{\blambda '}^k e_{\omega_k}\|^2
&= 1+\int_0^1 \Bigl( 1+\ldots+t^{k-1} \Bigr)\ud\tau_k(t)\\
&\geq  1+\int_{(0,1]} \Bigl(\frac{1}{t^k}+\ldots+\frac{1}{t}\Bigr)t^k \ud\tau_k(t)\\
&=1+ \bigl|\lambda'^{(k)}_\omega\bigr|^2
\int_{(0,1]} \Bigl( \frac{1}{t^k}+\ldots+\frac{1}{t} \Bigr)\ud\tau(t)\\
&> \bigl|\lambda'^{(k)}_\omega\bigr|^2
\int_0^1 \Bigl(\frac{1}{t^k}+\ldots+\frac{1}{t} \Bigr) \ud\tau(t).
\end{align*}
Dividing both sides by $\bigl|\lambda'^{(k)}_\omega\bigr|^2$ yields (b).

\pImp{b}{a}
Set $C_0:=\int_0^1 \bigl(\frac{1}{t}+\ldots+\frac{1}{t^k}\bigr) \ud\tau(t)$.
Fix an arbitrary real number $C\geq\frac{1}{1-C_0}$. Consider the completely alternating
sequence $\{ a_n \}_{n=0}^\infty$, where $a_n:= C\|S_\blambda^n e_\omega\|^2$.
Its representing measure is $C\tau$ and
\[ \int_0^1\Bigl(\frac{1}{t}+\ldots+\frac{1}{t^k}\Bigr)\ud C\tau(t)\leq a_0-1. \]
It follows from Lemma~\ref{lem:compalt-ext} that the sequence can be extended to
a completely alternating sequence $\{a_{n-k} \}_{n=0}^\infty$
with positive terms such that $a_{-k}=1$.

We can find weights $\lambda'_{\omega_l}$ for $l=0,\ldots,k-1$ such that
\begin{equation}\label{eq:back-ext1-che}
|\lambda'^{(j)}_{\omega_{k-j}}|^2 = \|S_{\blambda'}^j e_{\omega_k}\|^2= a_{j-k},\quad j=0,\ldots,k,
\end{equation}
where $\omega_0=\omega$, $\blambda'=\{\lambda'_v\}_{v\in V\sqlow{k}}$,
$\lambda'_{\omega_k}=0$ and $\lambda'_v=\lambda_v$ for $v\in V^\circ$.
This can be done e.g.\ by defining $\lambda'_{\omega_l}:= \sqrt{\frac{a_{-l}}{a_{-l-1}}}$
for $l=0,\ldots,k-1$. Clearly $S_\blambda \subset S_{\blambda'}$.

Observe that by \eqref{eq:back-ext1-che} we have
\[ |\lambda'^{(k)}_{\omega_0}|^2=a_0=C\|S^0_\blambda e_\omega\|^2=C,\]
and consequently
\[\|S_{\blambda '}^{k+n} e_{\omega_k}\|^2=
\|\lambda'^{(k)}_{\omega_0} S^n_{\blambda'}e_{\omega_0}\|^2=
C\|S_\blambda^n e_\omega\|^2=a_n,\quad n\geq 0.\]
This combined with \eqref{eq:back-ext1-che} yields
\[
	\|S_{\blambda '}^n e_{\omega_k}\|^2=a_{n-k},\quad n\geq 0,
\]
which implies that for $j=0,\ldots,k$ and $n\geq 0$,
\[ \|S_{\blambda '}^n e_{\omega_j}\|^2 \refoneq{eq:back-ext1-che}
\frac{\|S_{\blambda '}^n S_{\blambda '}^{k-j} e_{\omega_k}\|^2}{a_{-j}}=
\frac{a_{n-j}}{a_{-j}}. \]
This shows that the sequence $\{\|S_{\blambda '}^n e_{\omega_j}\|^2\}_{n=0}^\infty$
for $j=1,\ldots,k$ is completely alternating being a scaled tail of
the sequence $\{a_{n-k}\}_{n=0}^\infty$ (see Proposition~\ref{pro:alt-seq}).

Since 
$\{\|S_{\blambda '}^n e_v\|^2\}_{n=0}^\infty = \{\|S_\blambda^n e_v\|^2\}_{n=0}^\infty$
for $v\in V$, from Corollary~\ref{cor:shift-hypexp} we conclude that the constructed
$k$-step backward extension is completely hyperexpansive. 
\end{proof}

Isometric weighted shifts are very special case of the completely hyperexpansive
ones. They clearly admit completely hyperexpansive $k$-step backward extensions for
any $k\in \N$. In fact this makes them unique in the whole class.
\begin{Cor}
Let $S_\blambda$ be a bounded completely hyperexpansive proper weighted shift on
a rooted directed tree. Then \NWSR
\begin{abece}
\item $S_\blambda$ is an isometry,
\item $S_\blambda$ admits a completely hyperexpansive $k$-step backward extension
for every $k\in \N$.
\end{abece}
\end{Cor}
\begin{proof}
\pImp{a}{b} It is immediate consequence of Proposition~\ref{pro:trivial-back-ext}
and the fact that isometries are completely hyperexpansive.

\pImp{b}{a} Let $\tau$ be the representing  measure of a completely
alternating sequence $\{\|S_\blambda^n e_\omega\|^2\}_{n=0}^\infty$, where
$\omega$ is the root of underlying tree $(V,\parf)$.
From Theorem~\ref{thm:che-ext} we see that
\[ \int_0^1 n \ud\tau \leq \int_0^1 \Bigl( \frac{1}{t}+\ldots+\frac{1}{t^n} \Bigr) \ud\tau(t) <1,
\quad n\in\N. \]
This implies that $\tau\equiv 0$ and hence
\begin{equation}\label{eq:root-norm1}
\|S_\blambda^n e_\omega\|=1,\quad n\in\ZP.
\end{equation}
Complete hyperexpansivity of $S_\blambda$ implies that
\[
\sum_{u\in\Chif(v)} |\lambda_u|^2 = \|S_\blambda e_v\|^2
\geq \|e_v\|^2=1, \quad v\in V.\]
(In particular $\Chif(v)\neq \emptyset$ for all $v \in V$.) It suffices
to prove that $\|S_\blambda e_v\|^2 \leq 1$ for every $v\in V$.
Pick $v\in V=\Des(\omega)$. There exists unique $n\in\ZP$ such that
$v\in\Chif[n](\omega)$ (cf.\ Lemma~\ref{lem:tree-basics}~\eqref{tbs:kChi}).
Denote also $B:=\Chif[n](\omega)\setminus \{v\}$. Then by \eqref{eq:root-norm1},
\begin{align*}
\|S_\blambda^{n}e_\omega\|^2=\|S_\blambda^{n+1}e_\omega\|^2
&\refoneq{eq:normasumy}
\sum_{u\in\Chif[n](\omega)}|\lambda_u^{(n)}|^2\|S_\blambda e_u\|^2 \\
&\spacedeq{.9ex} \sum_{u\in B}|\lambda_u^{(n)}|^2\|S_\blambda e_u\|^2
+ |\lambda_v^{(n)}|^2\|S_\blambda e_v\|^2\\
&\spacedeq[\geq]{.9ex}
\sum_{u\in B}|\lambda_u^{(n)}|^2 + |\lambda_v^{(n)}|^2\|S_\blambda e_v\|^2\\
&\refoneq{eq:Sdok} \|S_\blambda^{n}e_\omega\|^2-|\lambda_v^{(n)}|^2
+ |\lambda_v^{(n)}|^2\|S_\blambda e_v\|^2\\
&\spacedeq{.9ex} \|S_\blambda^{n}e_\omega\|^2
+|\lambda_v^{(n)}|^2\bigl(\|S_\blambda e_v\|^2 -1\bigr).
\end{align*}
By properness, $|\lambda_v^{(n)}|^2>0$ and hence $\|S_\blambda e_v\|^2 -1 \leq 0$,
which completes the proof.
\end{proof}

\begin{Rem}\label{rem:proper-isometry}
Without the assumption of properness, we can easily construct a
non-isometric weighted shift admitting a completely hyperexpansive
$k$-step backward extension for arbitrary $k\in\N$. It suffices to
start from one isometric and one non-isometric completely hyperexpansive
weighted shift on disjoint rooted directed trees.
\begin{figure}[hbt]
\includegraphics[scale=1]{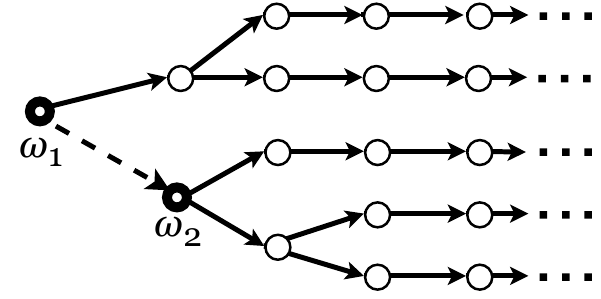}
\caption{Illustration to Remark~\ref{rem:proper-isometry}.}
\label{fig:klejenie}\end{figure}
Then modify the trees so that the root $\omega_2$ of the tree with non-isometric
weighted shift has a parent belonging to the other tree and preserve all the weights,
in particular $\lambda_{\omega_2}=0$ (see Figure~\ref{fig:klejenie}).
The weighted shift on the new tree is in fact an orthogonal sum of the original
weighted shifts.
\end{Rem}

Knowing the characterisation form Theorem~\ref{thm:che-ext} we are ready
to prove that the class of completely hyperexpansive bounded weighted
shifts on directed trees does not have the joint backward extension property.
This is shown in the following example.

\begin{Przyk}\label{ex:che-nonjoint}
Let $\Tree=(V_{2,0},\parf_2)$ be a directed tree consisting of two
infinite branches originating from the root (see $\mathscr{T}_{2,0}$ in
\cite[(6.2.10)]{szifty} and ``2-arm star'' in \cite[Definition~4.4]{pikul}).
Pick $k\in\N$ and $\alpha\in (0,\frac{1}{k})$ and define the weights
$\blambda=\{\lambda_v\}_{v\in V_{2,0}}$ by (cf.\ Figure~\ref{fig:che-nonjoint})
\[\lambda_0=0,\quad \lambda_{(j,n)}= \begin{cases}
\sqrt{\frac{n}{n-1}} & $if $j=1,\ n\geq 2\\
\sqrt{\alpha} & $if $ j=n=1\\
1 & $if $ j=2 
\end{cases},\quad (j,n)\in V_{2,0}. \]
Note that $\Tree$ is a rooted sum of the directed trees $\Tree_j:=\Tree_{((j,1)\to)}$
for $j=1,2$. The weighted shift $S_\blambda$ admits a completely hyperexpansive
$k$-step backward extension, while the weighted shift $S_{\blambda^\to_{(1,1)}}$
on $\Tree_1$ does not admit a $1$-step backward extension.
\end{Przyk}
\begin{figure}[hbt]
\includegraphics[scale=1.2]{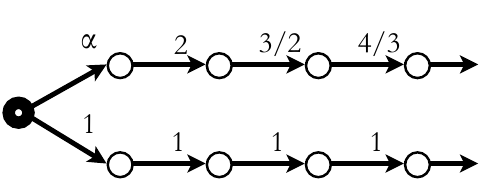}
\caption[The weighted shift from Example~\ref{ex:che-nonjoint}.]%
{The weighted shift from Example~\ref{ex:che-nonjoint}.
Number over each edge is the square of the corresponding weight.}
\label{fig:che-nonjoint}
\end{figure}

\begin{proof}
It can be easily verified that $\|S_\blambda^n e_{(1,1)}\|^2=n+1$ for $n\in\ZP$,
and hence the sequence $\{\|S_\blambda^n e_{(1,1)}\|^2\}_{n=0}^\infty$ is completely
alternating with the representing measure $\delta_1$ (see \eqref{eq:altmeasure0}).
Note that for any other vertex $v$ of $\Tree_{((1,1)\to)}$, the sequence 
$\{\|S_\blambda^n e_v\|^2\}_{n=0}^\infty$ is a scaled tail of
$\{\|S_\blambda^n e_{(1,1)}\|^2\}_{n=0}^\infty$, hence it is completely alternating
(see Proposition~\ref{pro:alt-seq}), and consequently the weighted shift
$S_{\blambda^\to_{(1,1)}}$ is completely hyperexpansive
(see Corollary~\ref{cor:shift-hypexp}).
According to Theorem~\ref{thm:che-ext}, $S_{\blambda^\to_{(1,1)}}$ does not admit
completely hyperexpansive $1$-step backward extension.

Clearly, $\|S_\blambda^n e_{(2,1)}\|^2=1$ for $n\in\ZP$, so the sequence
$\{\|S_\blambda^n e_{(2,1)}\|^2\}_{n=0}^\infty$ is completely alternating with
the zero representing measure. The equalities
\[ \|S_\blambda^n e_0\|^2=\alpha n + 1,\quad n\in\ZP, \]
imply that the sequence $\{\|S_\blambda^n e_0\|^2\}_{n=0}^\infty$ is completely
alternating with the representing measure $\alpha\delta_1$ and 
\[ \int_0^1 \Bigl(\frac{1}{t}+\ldots+\frac{1}{t^k}\Bigr)
 \ud\alpha\delta_1(t)=k \alpha < 1.\]
Using Corollary~\ref{cor:shift-hypexp} one can easily verify that $S_\blambda$ is
completely hyperexpansive and by Theorem~\ref{thm:che-ext} we obtain that it 
admits a completely hyperexpansive $k$-step backward extension.
\end{proof}

Even though the extendability of a weighted shift on the rooted sum does not
posses such elegant properties as in the case of subnormal and power hyponormal
operators, some positive results can be proven. For example, the existence of
$(k+1)$-step backward extensions for all members of a given family of
weighted shifts implies the existence of a $k$-step backward extension of
the enveloping operator on the rooted sum.

\begin{Thm}\label{thm:che-rootsum}
Let $J$ be a non-empty at most countable set and $k\in\ZP$.
For $j\in J$, let $S_j$ be a bounded weighted shift on a directed tree
$\Tree_j=(V_j,\parf_j)$ with root $\omega_j$. 
Assume that $S_j$'s are uniformly bounded and each of them
admits a completely hyperexpansive $(k+1)$-step backward extension.

Then there exists a system $\{ \theta_j \}_{j\in J}\subseteq \C\setminus\{0\}$
such that the weighted shift $S_\blambda$ on $\rootsum_{j\in J}\Tree_j=(V_J,\parf)$
with weights $\blambda =\{\lambda_v\}_{v\in V_J}$ determined~by
\begin{equation}\label{eq:weight-rest-che}
\lambda_{\omega_j}=\theta_j \text{ and } S_{\blambda^\to_{\omega_j}} = S_j
\text{ for all } j\in J,
\end{equation}
is bounded and admits a completely hyperexpansive $k$-step backward extension.
Moreover, if $k\geq 1$, then
\begin{equation}\label{eq:che-ext-bound}
\|S_\blambda\|\leq \max\biggl\{\sup_{j\in J}\|S_j\|,\, \sqrt{\frac{k+1}{k}}\biggr\}.\end{equation}
\end{Thm}

The following fact is going to be useful in the forthcoming
reasonings using representing measures. We skip its standard
measure-theoretic proof.
\begin{Lem}\label{lem:measure}
Let $\{\mu_j\}_{j\in J}$ be a family of finite measures on a fixed
measure space $(X, \mathcal A)$ and $\{a_j\}_{j\in J}$ be non-negative
real numbers such that $\sum_{j\in J} a_j\mu_j(X) <\infty$.
Then $\tilde\mu$ defined by
\[ \tilde\mu(A):=\sum_{j\in J} a_j\mu_j(A),\quad A\in\mathcal A, \]
is a finite measure on $(X, \mathcal A)$
and for every $\mathcal A$-measurable $f\colon X\to [0,+\infty]$
the following equality holds
\begin{equation}\label{eq:measure}
\int_X f \ud\tilde\mu = \sum_{j\in J} a_j \int_X f \ud\mu_j.
\end{equation}
\end{Lem}

\begin{proof}[Proof of Theorem~\ref{thm:che-rootsum}]
By Theorem~\ref{thm:che-ext}, the representing measure of the completely alternating
sequence $\{\|S_j^n e_{\omega_j}\|^2\}_{n=0}^\infty$, say $\tau_j$,
satisfies
\[ D_j:=\int_0^1\Bigl(\frac{1}{t}+\ldots+\frac{1}{t^{k+1}}\Bigr)\ud\tau_j(t)<1,
\quad j\in J.\]
Consequently, also
\[ C_j:=\int_0^1 \frac{1}{t}\ud\tau_j(t)\leq \frac{1}{k+1}D_j < \frac{1}{k+1},
\quad j\in J.\]
Therefore $\tau_j(\{0\})=0$ for $j\in J$.
Pick positive real numbers $\{a_j\}_{j\in J}$ such that\footnote{Except for the case
$k=0$, the inequality ($**$) follows from ($*$).}
\[ \mathrm{(*)}\ \sum_{j\in J} a_j (1-C_j)=1,\qquad
\mathrm{(**)}\ \sum_{j\in J} a_j < \infty. \]
Take $\{\theta_j\}_{j\in J}$ such that $|\theta_j|^2=a_j$ for $j\in J$.
Let $\blambda$ be as in \eqref{eq:weight-rest-che} with $\lambda_\omega=0$.
Then, by assumption,
\begin{equation}\label{eq:normy-rest}
\sum_{u\in \Chif(v)}|\lambda_u|^2 =\|S_j e_v\|^2\leq \sup_{j\in J}\|S_j\|^2 <\infty,
\quad j\in J,\ v\in V_j.
\end{equation}
It follows from ($**$) that $\sum_{u\in \Chif(\omega)}|\lambda_u|^2 <\infty$,
hence, by Proposition~\ref{pro:shift-basic}~\eqref{shbs:bounded}, the operator
$S_\blambda$ is bounded.

Define the Borel measure $\tau$ on $[0,1]$ by
\[ \tau(A):= \sum_{j\in J}a_j\int_A \frac{1}{t} \ud\tau_j(t),
\quad A\in\Borel([0,1]).\]
Since $C_j<1$ for $j\in J$, ($**$) implies that $\tau$ is a finite measure.

By Lemma~\ref{lem:measure} we get (recall that $\tau_j(\{0\})=0$ for every $j\in J$)
\begin{align*}
\|S_\blambda^{n+1} e_\omega\|^2
&\refoneq{eq:normasumy}\sum_{j\in J}a_j \|S_j^n e_{\omega_j}\|^2\\
&\refoneq{eq:altmeasure0} \sum_{j\in J} a_j\biggl(1+ \int_0^1 \bigl(1+\ldots+ t^{n-1}\bigr) \ud\tau_j(t)\biggr) \\
&\spacedeq{1ex} \sum_{j\in J} a_j\biggl(1 -\int_0^1\frac{1}{t}\ud\tau_j(t)+
\int_0^1 \bigl(1+\ldots+t^n\bigr)\frac{1}{t}\ud\tau_j(t)\biggr)\\
&\spacedeq{1ex} \sum_{j\in J} a_j (1 -C_j)
+\sum_{j\in J}a_j\int_0^1 (1+\ldots+t^n)\frac{1}{t}\ud\tau_j(t)\\
&\spacedeq[\stackrel{(*)}{=}]{.9ex} 1+\int_0^1 \bigl( 1+\ldots+t^n \bigr) \ud\tau(t),\quad n\geq 1.
\end{align*}
Observe also that
\[ \|S_\blambda e_\omega\|^2 = \sum_{j\in J}a_j
= 1+\sum_{j\in J}a_j C_j = 1+\int_0^1 t^0 \ud\tau(t). \]
Summarising, we have proven that $\{ \|S_\blambda^n e_\omega\|^2 \}_{n=0}^\infty$
is a completely alternating sequence with the representing measure $\tau$. 
For $j\in J$, $v\in V_j$ and $n\in\ZP$, we have
$\|S^n_\blambda e_v\|^2 = \|S_j^n e_v\|^2$,
and since $S_j$ is completely hyperexpansive,
$\{ \|S^n_\blambda e_v\|^2 \}_{n=0}^\infty$ is a completely alternating sequence.
According to Corollary~\ref{cor:shift-hypexp}, we see that $S_\blambda$ is
completely hyperexpansive. In the case $k=0$ the proof is complete.

Now suppose that $k\geq 1$. Then
\[
\sum_{j\in J} a_j = \sum_{j\in J}a_j (1-C_j) \frac{1}{1-C_j}
\stackrel{(*)}{<} \frac{1}{1-\frac{1}{k+1}}= \frac{k+1}{k}.
\]
This gives $\|S_\blambda e_\omega\|^2\leq \frac{k+1}{k}$, which combined with
\eqref{eq:normy-rest} proves the inequality \eqref{eq:che-ext-bound}.

Using Lemma~\ref{lem:measure} and the fact that $\tau(\{0\})=0$, we get
\begin{align*}\nonumber
\int_0^1 \Bigl(\frac{1}{t}+\ldots+\frac{1}{t^k}\Bigr)\ud\tau(t) \nonumber
&= \sum_{j\in J}a_j \int_0^1 \Bigl(\frac{1}{t}+\ldots+\frac{1}{t^k}\Bigr) \frac{1}{t}\ud\tau_j(t)\\
\nonumber &= \sum_{j\in J}a_j \int_0^1 \Bigl(\frac{1}{t^2}+\ldots+\frac{1}{t^{k+1}}\Bigr) \ud\tau_j(t)\\
\nonumber &= \sum_{j\in J}a_j (D_j - C_j)\\
&< \sum_{j\in J}a_j (1 - C_j) \overset{(*)}{=} 1. 
\end{align*}
This together with Theorem~\ref{thm:che-ext} proves that
$S_\blambda$ admits a completely hyperexpansive $k$-step backward extension.
\end{proof}

\begin{Rem}
Observe that the above theorem yields the implication
(c)$\Rightarrow$(a)$\land$(b) from Definition~\ref{dfn:ext-prop}
(JBEP) with $c_k=\sqrt{\frac{k+1}{k}}$. Since the proof of
(b)$\Rightarrow$(a) in Theorem~\ref{thm:joint-extension} does not use
the other part of JBEP, the implication is valid for the class of
completely hyperexpansive weighted shifts.
\end{Rem}

\begin{Rem}
Note that in the case of weighted shifts bounded from below the
implication (a)$\Rightarrow$(c) in the joint backward extension
property is still true, unlike the opposite one (see
Proposition~\ref{pro:below-fail}). This is a different way of
violating JBEP from what we have proven about
completely hyperexpansive weighted shifts.
\end{Rem}

As shown in Example~\ref{ex:che-nonjoint}, even if not all members of a given
family of weighted shifts admit completely hyperexpansive backward extensions
they may still be jointly extended to a rooted sum of the underlying trees.
In fact, at least one of them has to admit appropriate backward extension,
as described in the following theorem.

\begin{Thm}\label{thm:che-joint2}
Let $S_\blambda$ be a bounded proper weighted shift on a directed
tree  $\Tree$ with root $\omega$ admitting a completely hyperexpansive $k$-step
backward extension for some $k\in\ZP$.
Then for every $n\in\ZP$, there exists $v\in\Chif[n](\omega)$ such that
the weighted shift $S_{\blambda^\to_v}$ on $\Tree_{(v\to)}$ admits a
completely hyperexpansive $(n+k)$-step backward extension.
\end{Thm}
\begin{proof}
Since (see Definition~\ref{dfn:backward-ext})
\[ \Chit{n}{\Tree}(\omega)=\Chit{n+k}{\Tree\sqlow{k}}(\omega_k),\quad
n\in\ZP, \] we can reduce ourselves to the case when $k=0$,
replacing $S_\blambda$ by its completely hyperexpansive extension
on $\Tree\sqlow{k}$ if necessary.

We proceed by induction on $n\in\ZP$. Clearly, for $n=0$ the claim
is satisfied.

Assume $n\geq 0$ and there exists $u\in\Chif[n](\omega)$ such that
the weighted shift $S_{\blambda^\to_u}$ on $\Tree_{(u\to)}$ admits a
completely hyperexpansive $n$-step backward extension.
We want to prove that for some $v\in\Chif[n+1](\omega)$ the weighted
shift $S_{\blambda^\to_v}$ admits a completely hyperexpansive $(n+1)$-step
backward extension.
We will find such $v$ in $\Chif(u)\subseteq\Chif[n+1](\omega)$.
Denote by $\tau$ the representing measure of the completely alternating sequence
$\{\|S_\blambda^{j+1} e_u\|^2\}_{j=0}^\infty$ (see Lemma~\ref{lem:che-general}).
This sequence admits an extension to a completely alternating sequence
$\{\|S_\blambda^j e_u\|^2\}_{j=0}^\infty$, satisfying $\|S_\blambda^0 e_u\|^2=1$,
hence by Lemma~\ref{lem:compalt-ext}, $\tau(\{0\})=0$,
\begin{equation}\label{eq:int+norm}
\int_0^1 \frac{1}{t} \ud\tau(t) \leq \|S_\blambda e_u\|^2 - 1
\end{equation}
and the representing measure $\rho$ for the sequence
$\{\|S_\blambda^j e_u\|^2\}_{j=0}^\infty$ is given by the formula:
\[ \rho(A)=\int_A\frac{1}{t}\ud\tau(t) +\biggl(\|S_\blambda e_u\|^2
- 1-\int_0^1\frac{1}{t}\ud\tau(t)\biggr)\delta_0(A),\ \ A\in\Borel([0,1]).\]

For $v\in \Chif(u)$, let $\tau_v$ be the representing measure of the
completely alternating sequence $\{\|S_\blambda^j e_v\|^2\}_{j=0}^\infty$. 
We can define the Borel measure $\tau_0$ on $[0,1]$ as follows
\[ \tau_0(A) := {\textstyle \sum_{v\in\Chif(u)}}|\lambda_v|^2 \tau_v(A), 
\quad A\in\Borel([0,1]). \]
Observe that
\[ \tau_v([0,1])=\int_0^1 1\ud\tau_v
= \|S_\blambda e_v\|^2-1 \leq \|S_\blambda\|^2-1, \quad v\in \Chif(u).\]
This fact, combined with convergence of the series $\sum_{v\in\Chif(u)}|\lambda_v|^2$
gives that the measure $\tau_0$ is finite.
Then
\begin{align*}
\|S_\blambda^{j+1} e_u\|^2
&\,\refoneq{eq:normasumy}\, \sum_{v\in\Chif(u)}|\lambda_v|^2 \|S_\blambda^j e_v\|^2\\
&\,\refoneq{eq:altmeasure0}\, \sum_{v\in\Chif(u)}|\lambda_v|^2
\biggl(1+\int_0^1 \bigl(1+\ldots+t^{j-1}\bigr) \ud\tau_v(t) \biggr) \\
&\refoneq{eq:measure} \|S_\blambda e_u\|^2
+ \int_0^1 \bigl(1+\ldots+t^{j-1}\bigr)\ud\tau_0(t),
\quad j\geq 1.
\end{align*}
This means that $\tau_0$ is a representing measure for the completely alternating
sequence $\{ \|S_\blambda^{j+1} e_\omega\|^2 \}_{j=0}^\infty$
and, by Corollary~\ref{cor:alt-unique}, $\tau_0=\tau$.
This in particular means that $\tau_0(\{0\})=0$, and consequently $\tau_v(\{0\})=0$
for $v\in \Chif(u)$.
 
If $n\geq 1$, then we use the existence of a completely hyperexpansive $n$-step backward
extension of $S_{\blambda^\to_u}$ (\footnote{Note that $S_{\blambda^\to_u}^j e_u
= S_\blambda^j e_u$ for all $j\in\ZP$.}) and Theorem~\ref{thm:che-ext} to obtain that
$\rho(\{0\})=0$, which implies $\rho(A)=\int_A\frac{1}{t}\ud\tau(t)$ for any
$A\in\Borel([0,1])$. By the same theorem
\begin{equation}\label{eq:che-tau-k}
1> \int_0^1 \Bigl(\frac{1}{t}+\ldots+\frac{1}{t^n}\Bigr) \ud\rho(t)
= \int_0^1 \Bigl(\frac{1}{t^2}+\ldots+\frac{1}{t^{n+1}}\Bigr) \ud\tau(t).
\end{equation}
Set $D_v:=\int_0^1 \left( \frac{1}{t}+\ldots+\frac{1}{t^{n+1}} \right)\ud\tau_v(t)$
for $v\in\Chif(u)$. According to Theorem~\ref{thm:che-ext}, it suffices to prove that
$D_{v_0} <1$ for some $v_0\in\Chif(u)$.
Note that (cf.\ Lemma~\ref{lem:measure})
\begin{align*}
0&\spacedeq[\stackrel{(\dagger)}{>}]{1.3ex}\int_0^1 \Bigl(\frac{1}{t}+\ldots+\frac{1}{t^{n+1}}\Bigr) \ud\tau(t)
- \int_0^1 \frac{1}{t} \ud\tau(t) -1\\
&\refoneq[\geq]{eq:int+norm} \sum_{v\in\Chif(u)}|\lambda_v|^2
\int_0^1 \Bigl(\frac{1}{t}+\ldots+\frac{1}{t^{n+1}}\Bigr) \ud\tau_v(t)
-\sum_{v\in\Chif(u)}|\lambda_v|^2\\
&\spacedeq{1.3ex}\sum_{v\in\Chif(u)}|\lambda_v|^2 (D_v - 1),
\end{align*}
where ($\dagger$) is a consequence of \eqref{eq:che-tau-k} if $n\geq 1$ and is
obvious for $n=0$.
At least one of the summands $|\lambda_v|^2 (D_v - 1)$ has to be negative,
hence, by properness of $S_\blambda$, $D_{v_0} <1$ for some $v_0\in\Chif(u)\subseteq\Chif[n+1](\omega)$.
This completes the proof.
\end{proof}

We finish the article listing all classes of weighted shifts on directed trees
for which we have decided whether they satisfy JBEP:\medskip

{\centering\begin{tabular}[hbt]{|l|l|}\hline
Classes having JBEP \quad \rule{0pt}{3ex}& Classes without JBEP
\\\hline
all bounded \rule{0pt}{3ex} & bounded from below\\
isometric & quasinormal  \\
contractive  & completely hyperexpansive\\
expansive & \\
hyponormal & \\
power hyponormal & \\
subnormal & \\ \hline
\end{tabular}\par}

\subsection*{Acknowledgements}
Results presented in this paper are a substantial part of my
dissertation written under supervision of professor Jan Stochel,
to whose guidance I am very grateful.


\end{document}